\documentclass[epic,eepic,11pt]{amsart}

\usepackage{amsfonts,mathrsfs}
\usepackage[mathscr]{eucal}

\usepackage{amssymb}
\usepackage{amscd}
\usepackage{fancyhdr}
\usepackage[all,cmtip]{xy}

\usepackage{euler,eucal}

\DeclareMathAlphabet{\mathbf}{T1}{ppl}{bx}{n}
\DeclareMathAlphabet{\mathrm}{T1}{ppl}{m}{n}

\numberwithin{equation}{section}

\newcommand\note[1]%
{$^\dagger$\marginpar{\footnotesize{$^\dagger${#1}}}}

\def\({\left(}
\def\){\right)}
\def\<{\left<}
\def\>{\right>}

\newtheorem{theorem}{Theorem}[section]
\newtheorem{proposition}[theorem]{Proposition}
\newtheorem{lemma}[theorem]{Lemma}
\newtheorem{definition}[theorem]{Definition}

\newtheorem{corollary}[theorem]{Corollary}

\theoremstyle{definition}
\newtheorem{example}[theorem]{Example}
\newtheorem{remark}[theorem]{Remark}

\newcommand\lie{\mathfrak}
\renewcommand\t{\lie{t}}
\newcommand\g{\lie{g}}

\newcommand{\crit}{\mathrm{crit}}
\newcommand\bb[1]{{\text{\bf#1}}}

\newcommand\Z{\bb{Z}}

\newcommand\R{\mathbb{R}}
\newcommand\C{\mathbb{C}}

\newcommand\J{\mathcal{J}}
\newcommand\cO{\mathcal{O}}

\newcommand     {\comment}[1]   {}
\newcommand{\mute}[2] {}
\newcommand     {\printname}[1] {}

\newcommand\func[1]{\operatorname{\mathrm{#1}}}
\newcommand\funclim[1]{\operatorname*{\mathrm{#1}}}

\renewcommand\deg{\func{deg}}

\renewcommand\dim{\func{dim}}

\renewcommand\exp{\func{exp}}

\newcommand\im{\func{im}}

\renewcommand\inf{\func{inf}}

\renewcommand\ker{\func{ker}}
\renewcommand\lim{\funclim{lim}}

\renewcommand\sup{\func{sup}}

\newcommand\sur{\mathrel{\to\kern-1.8ex\to}}
\newcommand\iso{\mathrel{\hookrightarrow\kern-1.8ex\to}}

\newcommand\longhookrightarrow{\lhook\joinrel\longrightarrow}

\newcommand\longsur{\mathrel{\longrightarrow\kern-1.8ex\to}}
\newcommand\longiso{\mathrel{\longhookrightarrow\kern-1.8ex\to}}

\renewcommand\supset{\supseteq}

\begin{document}

\bibliographystyle{amsalpha}
\date{\today}
\title{Topology of generalized complex quotients}

\author{Thomas Baird, Yi Lin}
\maketitle

     {\Small         Dedicated to Prof. Victor Guillemin on the occasion of his seventieth birthday.}

\begin{abstract} Consider the Hamiltonian action of a torus on a compact twisted generalized complex manifold $M$. We first
observe that Kirwan injectivity and surjectivity hold for ordinary equivariant cohomology in this setting. Then
we prove that these two results hold for the twisted equivariant cohomology as well.

\end{abstract}
\section{Introduction}

Kirwan injectivity and surjectivity are two important results in equivariant symplectic geometry. Recall that for a symplectic manifold $(M, \omega)$, an action by a connected Lie group $G$ on $(M,\omega)$ is called Hamiltonian if it is regulated by a moment map $\mu: M \rightarrow \lie{g}^*$ taking values in the dual of the Lie algebra of $G$. Contracting by $\xi \in \lie{g}$ produces a real valued function $\mu^{\xi}:M \rightarrow \R$ called a component of the moment map. If $G$ is compact, then for any $\xi \in \lie{g}$, $\mu^{\xi}$ is a Morse-Bott function and may be used to study the equivariant topology of $M$. We will mostly focus on the case that $G =T$ is a torus.

In \cite{Kir86}, using ideas of Atiyah-Bott \cite{AB82}, Kirwan demonstrated that a Hamiltonian action on a compact symplectic manifold $M$ is equivariantly formal.  In particular, the equivariant cohomology of $M$ with rational coefficients satisfies a noncanonical isomorphism $$H_G^*(M) \cong H^*(M) \otimes H^*(BG)$$ as graded $H^*(BG)$-modules, where $BG$ is the classifying space for $G$. Furthermore, if $G=T$ is a torus, and $i: M^T \hookrightarrow M$ denotes inclusion of the fixed point set, the localization map in equivariant cohomology $i^*: H_T^*(M) \rightarrow H_T^*(M^T)$ is an injection, a result known as Kirwan injectivity. Her proof uses Morse theory of a component $\mu^{\xi}$ of the moment map.

Kirwan also showed that the map $\kappa: H_G(M) \rightarrow H_G(\mu^{-1}(0))$ induced by inclusion is a surjection. This result is known as Kirwan surjectivity and the map $\kappa$ is known as the Kirwan map. If $0$ is a regular value of $\mu$, then $H_G(\mu^{-1}(0)) \cong H(M // G)$, where $M//G = \mu^{-1}(0)/G$ is the symplectic quotient, so $H(M//G)$ is describable as a quotient ring $H_G(M)/ \ker(\kappa)$.

Kirwan's original proof of surjectivity involved studying the Morse theory of the norm square of the moment map $|| \mu ||^2$, which has minimum $\mu^{-1}(0)$.  In fact $|| \mu||^2$ is not Morse-Bott, but instead satisfies Kirwan's minimal degeneracy condition, which allows the basic constructions of Morse theory to be carried out.

Modern proofs of Kirwan surjectivity have avoided these technicalities. In \cite{TW98}, Tolman and Weitsman computed the kernel of $\kappa$ for torus actions using the honest Morse-Bott functions $\mu^{\xi}$, rather than $||\mu||^2$, the principle being that the kernel of $\kappa$ is built up of contributions from each circle in the torus. In \cite{Gol02}, Goldin used their ideas to produce a simplified proof of Kirwan surjectivity for torus actions, using circles actions and reduction in stages.  

Goldin's proof contains a gap, which was resolved by Ginzburg-Guillemin-Karshon  (\cite{GGK02} appendix G). They introduce the notion of a nondegenerate abstract moment map, which abstracts the relevant Morse-theoretic properties from the symplectic case, and prove Kirwan's theorems in this general setting.

In this paper, we generalize Kirwan injectivity and surjectivity to Hamiltonian actions on compact generalized complex manifolds, in the sense of Lin-Tolman \cite{LT05}. Generalized complex (GC) manifolds were introduced by Hitchin in \cite{H02} and developed by Gualtieri in his thesis \cite{Gua03}. They form a common generalization of both complex and symplectic manifolds and so are well suited to the study of Mirror Symmetry and conformal field theory. Generalized complex manifolds can also incorporate a twist by a closed 3-form $H \in \Omega^3(M)$. When $H$ is integral it may be interpreted as the curvature of a gerbe over $M$ and is known in the physics literature as the Neveu-Schwartz 3-form flux.

In the presence of a twisting $H$, it becomes interesting to study the twisted de Rham cohomology of $M$, $H(M; H)$, which is defined to be the cohomology of the complex consisting of the usual de Rham forms $\Omega(M)$ with a twisted differential $d+ H\wedge$. For example, Gualtieri \cite{Gua04} showed that for an $H$-twisted generalized K\"ahler manifold $M$, $H(M; H)$ inherits a Hodge decomposition and in Kapustin-Li \cite{KL04}, $H(M;H)$ is identified as the BRST cohomology of states for the associated conformal field theory.

In \cite{LT05}, Lin-Tolman extended the notion of Hamiltonian actions and reduction in symplectic geometry to the realm of generalized complex geometry. In the presence of a twisting 3-form $H \in \Omega^3(M)$,  their construction involves a generalized moment map $\mu: M \rightarrow \lie{g}^*$ and a moment 1-form $\alpha \in (\Omega^1(M) \otimes \lie{g}^*)^G$ for which $H+\alpha$ is an equivariantly closed 3-form in the Cartan model (c.f. \cite{GS99} ).
This construction turns out to be something very natural in physics. It has been shown by Kapustin and Tomasiello \cite{KT06}
that the mathematical notion of Hamiltonian actions on twisted generalized K\"ahler manifolds is in perfect agreement with the physical notion of general $(2,2)$ gauged sigma models with three-form fluxes.

 Inspired by Atiyah-Segal \cite{AS05}, the second author in  \cite{Lin07} used $H+\alpha$ to define the twisted equivariant cohomology, $H_T(M;\eta+\alpha)$. The basic properties of the twisted equivariant cohomologies were studied in \cite{Lin07} using Hodge theory, especially in the case of Hamiltonian actions on compact generalized K\"ahler manifolds. In the current paper we study the twisted equivariant cohomology using Morse theory, in the more general case of Hamiltonian actions on compact generalized complex manifolds. The following proposition is crucial to our paper.

\begin{proposition} \label{moment-map-Morse-Bott}
Consider the Hamiltonian action of a (compact) torus $T$ on a compact twisted GC-manifold $M$ with a generalized moment map $\mu: M \rightarrow \lie{t}^*$. Then $\mu$ is a nondegenerate abstract moment map in the sense of \cite{GGK02} (see Definition \ref{absmom}).
\end{proposition}

Proposition \ref{moment-map-Morse-Bott} paraphrases a result from Nitta's very interesting recent work \cite{NY07}. Nitta's result was known by the authors to hold under additional hypotheses but his general result came out as a welcome surprise. Because Nitta's theorem is central to our work, we provide a self-contained proof of Proposition \ref{moment-map-Morse-Bott} in Section 6. Our proof is a variation of Nitta's which has the advantages of being somewhat simpler and of extending to some examples of noncompact manifolds $M$. One key ingredient in Nitta's proof is the maximum principle for pseudo-holomorphic functions on almost complex manifolds, for which we provide a proof in Appendix \ref{maximum-principle}.

In view of the above proposition, results from \cite{GGK02} prove that Kirwan injectivity and surjectivity hold for ordinary equivariant cohomology. In our paper, we use parallel arguments to prove \emph{twisted} versions of equivariant formality, Kirwan injectivity, and Kirwan surjectivity:

\begin{theorem}[Equivariant formality]\label{eqformality}
Consider the Hamiltonian action of a compact connected group $G$ on a compact $H$-twisted generalized complex manifold $M$. Then we have a non-canonical isomorphism $$H_G(M;H +\alpha) \cong H(M;H)\otimes H(BG),$$ where $\alpha$ is the moment one form of the Hamiltonian action.
\end{theorem}

\begin{theorem}[Kirwan injectivity]\label{kirinj}
Let $T$ be a compact torus and let $M$ be a compact $H$-twisted generalized Hamiltonian $T$-space with induced equivariant 3-form $H +\alpha$, and let $i : M^T \rightarrow M$ denote the inclusion of the fixed point set. Then the induced map $$ i^*: H_T(M;H +\alpha) \rightarrow H_T(M^T;H+\alpha) \cong H(M^T;H)\otimes H(BT)$$ is an injection.
\end{theorem}

\begin{theorem}[Kirwan surjectivity]\label{kirsur}
Let $M$ be a compact $H$-twisted generalized Hamiltonian $T$-space with induced equivariant 3-form $H +\alpha$ and moment map $\mu$, where $T$ is a compact torus. For $c \in \lie{t}^*$ a regular value of $\mu$ we have:\begin{equation}
H_T(M;H +\alpha) \rightarrow H(\mu^{-1}(c)/ T ; \tilde{H})
\end{equation}
is a surjection, where $\tilde{H}$ is the twisting $3$-form inherited through reduction.
\end{theorem}

These results are established more generally for compact nondegenerate abstract moment maps with \emph{compatible}  equivariantly closed 3-form. We expect that Kirwan surjectivity remains true for the Hamiltonian action of a compact connected Lie group on a compact twisted generalized complex manifold, and we hope to return to this question in a later work.

Non-symplectic examples of Hamiltonian torus actions on generalized complex manifolds to which our results may be applied have been constructed in \cite{Lin07}
and \cite{Lin07b}. New examples constructed using surgery on toric varieties will be included in a forthcoming paper by the authors. We would also like to mention that a stronger version of Theorem \ref{eqformality} was previously proven for the case of Hamiltonian action of a compact Lie group on a generalized K\"ahler manifold in \cite{Lin07}, using the $\partial \bar{\partial}$ lemma and generalized Hodge theory.

We discuss now one possible application of our results. Suppose $(M,\J)$ is a twisted generalized complex manifold with a Hamiltonian $G$ action, and suppose $L$ is the $\sqrt{-1}$-eigenbundle of
$\J$. Then $L$ has a natural Lie algebroid structure, c.f. \cite{Gua03}. Moreover, the existence of a Hamiltonian $G$ action induces
a Lie algebra map $\frak{g} \rightarrow C^{\infty}(L)$. So there is an equivariant version of the Lie algebroid cohomology associated to the Hamiltonian $G$ action, in the sense of \cite{BCRR05}. The twisted equivariant cohomology studied in the current paper is closely related to the equivariant Lie algebroid cohomology. Indeed, they are canonically isomorphic to each other if $M$ is a generalized Calabi-Yau manifold satisfying the $\overline{\partial}\partial$-lemma. It is well known that information on the deformation of generalized complex structures is contained in the Lie algebroid cohomology of $L$. Therefore, the results established in this paper may indicate a close relationship between the deformation theory of the generalized complex manifold $M$ and that of its generalized complex quotients.


The layout of the paper is as follows. Section \ref{t-eq-cohomology} reviews twisted equivariant cohomology and proves a few lemmas for later use. Section \ref{morese theory} uses Morse theory to prove twisted Kirwan injectivity and surjectivity for nondegenerate abstract moment maps. Section \ref{g.c.geometry.review} gives a quick review of generalized complex geometry. Section \ref{g-moment-map} recalls the definition of generalized moment maps and proves Proposition \ref{moment-map-Morse-Bott}. Section \ref{main-results} establishes the main results of this paper, namely, Theorem \ref{eqformality}, \ref{kirinj},  \ref{kirsur}. Appendix \ref{maximum-principle} proves the maximum principle for pseudoholomorphic functions on almost complex manifolds. Appendix \ref{nondegabs} establishes a key Lemma about nondegenerate abstract moment maps postponed from \S \ref{morese theory}. Appendix \ref{apptwis} compares several versions of twisted equivariant cohomology existed in the literature. Appendix \ref{commalg} collects some commutative algebra results that we make frequent use of throughout, but particularly in Section \ref{t-eq-cohomology}.

\,\,\,\,\,\,\,\,\,\,\,\,\,\,\,\,\,\,\,\,\,\,\,\,\,\,\,\,\,\,\,\,\,\,\,\,\,\,\,\,\,\,\,\,\,\,
{\bf Acknowledgement}: The authors would like to thank Marco Gualtieri, Yael Karshon, Eckhard Meinrenken, Paul Selick, and Ping Xu for useful discussions. Y.L. is grateful to Lisa Jeffrey for providing him a Postdoctoral Fellowship at the University of Toronto, where he started his work on the equivariant cohomology theory of GC manifolds.

\section{Review of twisted equivariant cohomology}\label{t-eq-cohomology}

In this section we review twisted equivariant cohomology, as developed in Atiyah-Segal \cite{AS05}, Hu-Uribe \cite{HuUribe06}, Freed-Hopkins-Teleman \cite{FHT02} and Lin \cite{Lin07}.

\subsection{Definitions}
Let $G$ be a compact connected Lie group with Lie algebra $\lie{g}$ and dual $\lie{g}^*$. For $M$ a smooth $G$-manifold, we denote by $\xi_M$ the vector field on $M$ generated by $\xi \in \lie{g}$. The equivariant de Rham complex $(\Omega_G(M), d_G)$ is a differential graded (super)commutative algebra associated to the
$G$-manifold $M$. Here, $$\Omega_G(M) = (\Omega(M) \otimes S \lie{g^*})^G$$ is the space of polynomial functions on $\lie{g}$ taking values
in the space of differential forms $\Omega(M)$, which are equivariant under the induced $G$-action on $\Omega(M)$ and the adjoint action on the symmetric algebra $S \lie{g}^*$, and $d_G$ is defined by extending
linearly the formula $$(d_G( \sigma \otimes P))(\xi) = d \sigma \otimes P(\xi)- \iota_{\xi_M} \sigma \otimes P(\xi)$$ where $\sigma \in \Omega(M), P \in S \lie{g}^*$,
and $\xi \in \lie{g}$.  It comes equipped with a grading
$$ \Omega_G^n(M) = \bigoplus_k ( \Omega^{n-2k}(M) \otimes S^k \lie{g}^*)^G.$$
The equivariant de Rham complex computes the (Borel) equivariant cohomology of $M$ with real coefficients (we refer to \cite{GS99} for more details).

\begin{example}
When $G$ is trivial $(\Omega_G(M), d_G)$ is the usual de Rham complex.
\end{example}

\begin{example}
In the special case that $G=T$ is a compact torus with lie algebra $\lie{t}$,  $T$ acts trivially on $\lie{t}$ we have:$$\Omega_T(M) = \Omega(M)^T \otimes S \lie{t^*}$$ where $\Omega(M)^T$ is the space of $T$-invariant differential forms.
\end{example}

Let $ \hat{\Omega}_G(M)$ denote the direct product $\prod \Omega_G^i(M)$. The differential $d_G$ extends in a natural way to $\hat{\Omega}_G(M)$ and we adopt the convention that $H_G(M)$ is defined to be the cohomology of the complex $(\hat{\Omega}_G(M), d_G)$ (where we have abusively reused $d_G$ to denote its extension to $\hat{\Omega}_G(M)$). It follows that

\begin{equation}\label{homprod}
 H_G(M) := \prod_{i=0}^{\infty} H^i_G(M)
\end{equation}
as opposed to the more conventional direct sum. In the untwisted setting this is not a serious modification, but once twisting is introduced the direct product is much easier to work with. In this context, the equivariant cohomology of a point is $(\hat{S} \lie{g}^*)^G$, the ring of $G$-invariant formal power series on $\lie{g}$.


Given a $d_G$-closed 3-form, $\eta \in \Omega^3_G(M)$, we define a twisted differential $$d_{G,\eta} = d_G + \eta \wedge$$ on $\hat{\Omega}_G(M)$. Because $\eta$ is closed and of odd degree, it follows that $d^2_{G,\eta}=0$ and we define the \textbf{$\eta$-twisted equivariant cohomology}
\[ H_G(M;\eta)=\text{ker} d_{G,\eta}/\text{im}d_{G,\eta}.\]
Because $d_{G,\eta}$ is an odd operator, $H(M;\eta)$ inherits a $\Z_2$-grading from the $\Z$-grading on $\Omega_G(M)$. Because $d_{G,\eta}$ is usually not a derivation, $H_G(M;\eta)$ is usually not a ring, but is instead a module for the untwisted equivariant cohomology ring $H_G(M)$ and hence also for $\hat{S} \lie{g}^*$.

\begin{remark}
The cochain complex we are using to define twisted cohomology differs from those found in \cite{FHT02} and \cite{HuUribe06}, but gives rise to a naturally isomorphic cohomology theory (see Appendix \ref{apptwis}).
\end{remark}

\begin{remark}
We show in Appendix \ref{apptwis} that for a compact manifold $M$ the completion of twisted equivariant cohomology is obtained by extension of scalars from the uncompleted version, i.e.
$$ H_G(M ; \eta) \cong H( \Omega_G(M), d_{G,\eta}) \otimes_{(S\lie{g}^*)^G} (\hat{S}\lie{g}^*)^G$$
\end{remark}

\begin{example}\label{trivialact}
Suppose $G$ acts trivially on $M$. Then any $d$-closed 3-form $ \eta \in \Omega^3(M)$ determines a $d_G$-closed 3-form $ \eta \otimes 1 \in \Omega_G^3(M)$. In this case it is easy to see that $H_G(M;\eta \otimes 1) \cong H(M;\eta) \otimes (\hat{S}\lie{g}^*)^G$ canonically as $(\hat{S}\lie{g}^*)^G$-modules.
\end{example}

\begin{example}\label{trivialact2}
More generally, let $G$ act on $M$ such that a normal subgroup $H \subset G$ acts trivially. Then there is an induced $G/H$ action on $M$ and a chain isomorphism $ \Omega_{G/H}(M) \otimes (S \lie{h}^*)^H \cong \Omega_{G}(M)$. If $\eta \in \Omega_{G/H}^3(M)$ is $d_{G/H}$-closed, then $$ H_{G}(M;\eta \otimes 1) \cong H_{G/H}(M; \eta) \otimes (\hat{S} \lie{h}^*)^H $$ canonically.
\end{example}

We may consider a more general class of twisted complexes using the notion of differential graded modules.
Let $(C^*,\delta) = (\oplus_{k \geq 0} C^k, \delta)$ be a cochain complex graded by the nonnegative integers.  We say that  $(C^*,\delta)$ is a (left) $(\Omega^*_G(M), d_G)$-module, or simply a $\Omega_G(M)$-module, if $C^*$ is a graded module of the graded algebra $\Omega_G^*(M)$ and for all $\alpha \in \Omega_G(M)$ of pure degree and $x \in C^*$, the differential satisfies the identity:
\begin{equation*}
\delta ( \alpha \wedge x) = d_G(\alpha) \wedge x + (-1)^{\deg \alpha} \alpha \wedge \delta(x).
\end{equation*}
The differential $\delta$ extends naturally to a differential on $ \hat{C} := \prod_i C^i$, which by abuse of notation we also call $\delta$.
A closed 3-form $\eta \in \Omega_G(M)$ determines a twisted differential $ \delta_{\eta}:= \delta + \eta \wedge $ on $\hat{C}$ and we define \[H_G(C^*; \eta)=\text{ker} \delta_{\eta}/\text{im}\delta_{\eta}.\] The module structure descends to make $H_G(C^*;\eta)$ a $\Z_2$-graded module for $H_G(M)$.

\begin{example}\label{pairs}
Let $i: A \subset N$ a pair of embedded submanifolds of $M$ preserved by $G$. We use the algebraic mapping cone to define the differential graded complex $(\Omega_G^*(N,A), \delta)$ by
 \[  \Omega_G^n(N,A)= \Omega_G^{n+1}(N) \oplus \Omega_G^{n}(A) \]
with differential $\delta (n, a) =  (-d_G(n), d_G(a) + i^*(n))$ . Then $\Omega_G(N,A)$ is a $(\Omega_G(M), d_G)$-module under the action $x \wedge (n,a) = ( x \wedge n, x \wedge a)$. For $\eta \in \Omega_G^3(M)$ closed, we will use notation: \[H_G(N,A;\eta) = H(\Omega_G(N,A);\eta)\]
Notice that $H_G(N,A ; \eta) = H_G(N,A; j^*\eta)$, where $j: N \rightarrow M$ is the inclusion map.
\end{example}

\subsection{Basic Properties}Twisted cohomology is invariant under quasiisomorphism.

\begin{proposition}\label{doubledip}
Let $\phi: (C^*, \delta) \rightarrow (D^{*+n}, \delta')$ be a degree $n$ quasiisomorphism of $(\Omega_G(M), d_G)$-modules. Then the induced map $ H_G^*(C^*; \eta) \rightarrow H_G^{*+n}(D^*;\eta)$ is a degree $n~mod~2$ isomorphism for all $d_G$-closed $\eta \in \Omega_G^3(M)$.
\end{proposition}

\begin{proof}
Using the algebraic mapping cone construction, it suffices to prove that if $(C^*, \delta)$ is acyclic, then so is $(\hat{C}^*, \delta_{\eta})$.

Let $c = c_i + c_{i+1} + c_{i+2}  + ... \in \hat{C}^*$ be $\delta_{\eta}$-closed, where $c_k \in C^k$. Then necessarily $\delta (c_i) = 0$. By acyclicity, there exists $b_{i-1} \in C^{i-1}$ such that $\delta (b_{i-1}) = c_i$, so $$c - \delta_{\eta}(b_{i-1}) = c_{i+1} +(c_{i+2} - \eta \wedge b_{i-1}) + ...$$ has lowest degree term lying in $C^{i+1}$. Iterating the process, we can construct $b = b_{i-1} +b_i + ...$ satisfying $ \delta_{\eta} (b) = c$.
\end{proof}
It follows that many important properties of untwisted equivariant cohomology, such as homotopy invariance and excision, extend to twisted cohomology.
\begin{proposition}
Let $G$ be compact connected with maximal torus $T$ and Weyl group $W = N(T)/T$. For any $G$-manifold $M$ and twisting $\eta \in \Omega_G^3(M)$, we have a natural isomorphism $$ H_G(M;\eta) \cong H_T(M;\eta')^W$$ where $\eta' \in \Omega_T(M)^W$ is the image of $\eta$ under the map $\Omega_G(M) \rightarrow \Omega_T(M)$ induced by restricting the action.
\end{proposition}

\begin{proof}
The map $ \Omega_G(M) \rightarrow \Omega_T(M)$ restricts to a quasiisomorphism $ \Omega_G(M) \rightarrow  \Omega_T(M)^W$ which is also a $\Omega_G(M)$-module homomorphism in the obvious way. Thus by Proposition \ref{doubledip}, $$ H_G(M;\eta) \cong H_T( \Omega_T(M)^W ; \eta) \cong H_T(M; \eta')^W$$
\end{proof}
This result helps justify our later focus on torus actions.

Recall that in untwisted equivariant cohomology, we have the isomorphism $$\phi: H_G(M) \cong H(M/G),$$ provided that
the action of $G$ on $M$ is free. We have the following generalization.

\begin{proposition}(\cite{Lin07}, A.4.) \label{freeaction}
Let $M$ be a smooth $G$-manifold upon which $G$ acts freely and suppose $\dim H(M) < \infty$.  For $d_G$-closed $\eta \in \Omega_G^3(M)$ we have isomorphisms $$ H_G(M;\eta) \cong H(M/G;\bar{\eta})$$ where $\bar{\eta} \in \Omega^3(M/G)$ satisfies $\phi([\eta]) = [\bar{\eta}] \in H(M/G)$.
\end{proposition}

Given a short exact sequence $0 \rightarrow C^* \rightarrow D^* \rightarrow E^*\rightarrow 0$ of $\Omega_G(M)$-modules, twisting by $\eta$ gives rise to a six term exact sequence in the twisted cohomology.

\begin{example}
Recall the notation of Example \ref{pairs}. The pair $i: A \hookrightarrow N$ gives rise to a short exact sequence of $\Omega_G(M)$-modules, $ 0 \rightarrow \Omega_G(A) \rightarrow Cyl( i^*) \rightarrow \Omega_G(N,A) \rightarrow 0$, where the algebraic mapping cylinder, $Cyl(i^*)$, is naturally quasiisomorphic to $\Omega_G(N)$. We obtain a six term exact sequence:
\begin{equation*}\begin{CD}
 H^0_G(N,A;\eta) @>>> & H^0_G(N;\eta)@>>> H^0_G(A;\eta) \\
 @AAA      &                @.              @VVV \\
H^1_G(A;\eta) @<<<  & H^1_G(N;\eta) @<<<   H^1_G(N,A;\eta)
\end{CD}\end{equation*}
\end{example}

The next lemma shows that up to (noncanonical) isomorphism, the $\eta$-twisted equivariant cohomology depends only on the cohomology class $[\eta] \in H_G^3(M)$.

\begin{lemma}\label{carumba}
Let $b\in \Omega_G^2(M)$ be an equivariant 2-form and let $exp(b) = \sum_{i=0}^{\infty} b^n/ n!$. Then for any $(\Omega_G(M),d_G)$-module $(C^*,\delta)$, wedging by $exp(b)$ determines an isomorphism of chain complexes, $$ exp(b)\wedge  (\,\cdot\,) : (\hat{C}, \delta_{(\eta+d_Gb)}) \rightarrow (\hat{C}, \delta_{\eta})$$ which in particular determines an isomorphism $H_G(C^* ; \eta + d_Gb) \cong H_G(C^*; \eta )$.
\end{lemma}

\begin{proof}
The map $exp(b) \wedge (\,\cdot\,)  : \hat{C} \rightarrow \hat{C}$ is certainly even and linear.  The equations $ exp(b) \wedge exp(-b) = exp(-b) \wedge exp(b) = id_{\hat{C}}$ imply that $exp(b) \wedge (\,\cdot\,)$ is an isomorphism of vector spaces. Finally, for $\alpha \in \hat{C}$ we have \[\begin{split} \delta_{\eta}(exp(b)\wedge \alpha) = d_G( exp(b)) \wedge \alpha + exp(b) \wedge \delta_{\eta} \alpha \\= exp(b) \wedge d_Gb \wedge \alpha + exp(b) \wedge \delta_{\eta} \alpha = exp(b) \delta_{(\eta + d_Gb)}\alpha\end{split}\] so $exp(b) \wedge (\,\cdot\,)$ respects differentials.
\end{proof}

We may now state and prove the Thom isomorphism, which is due to Hu-Uribe \cite{HuUribe06}.

\begin{proposition}\label{tocome}
Let $\pi: E \rightarrow N$ be an orientable real vector bundle of rank $r$, let $i: N \rightarrow E$ denote inclusion as the zero section. Let $G$ be a compact torus acting on $E$ by bundle automorphisms, inducing an action on $N$. Let $\eta \in \Omega_G^3(E)$ be a $d_G$-closed form and let $\tau \in \Omega_G^r(E,E-N)$ be a $d_G$-closed form representing the usual equivariant Thom class (c.f. \cite{GS99}). Then the composition:

\begin{equation*}
H_G(N; \eta) \xrightarrow{\pi^*} H_G(E;\pi^*i^*\eta) \xrightarrow{\wedge \tau} H_G(E, E - N; \pi^* i^* \eta) \xrightarrow{ exp(b) \wedge } H_G(E, E-N; \eta)
\end{equation*}
is an isomorphism of degree $(r~ mod~2)$, where $b \in \Omega^2_G(E)$ satisfies $d_G(b) = \pi^* i^* (\eta) - \eta$.
\end{proposition}

\begin{proof}
By homotopy invariance $\pi^*$ is an isomorphism. The map $\wedge \tau : (\Omega_G(E), d_G) \rightarrow (\Omega_G(E, E-N),d_G)$ is a $\Omega_G(E)$-module morphism and induces a degree $r$ isomorphism $H_G^*(E) \cong H_G^{*+r}(E,E-N)$, so by Proposition \ref{doubledip}
\[ H_G(E;\eta) \xrightarrow{\wedge \tau} H_G(E, E - N; \pi^* i^* \eta)\] is also an isomorphism.
Finally, $\exp(b) \wedge $ is an isomorphism by Lemma \ref{carumba}.
\end{proof}

The equivariant Euler class plays the same role for twisted equivariant cohomology as it does for untwisted equivariant cohomology.

\begin{lemma}\label{euler}
Let $\pi: E \rightarrow N$ satisfy the hypotheses and notation of Proposition \ref{tocome}, and let $Eul_G(E) \in H_G^r(N)$ denote the equivariant Euler class of $E$. Then the following diagram is commutative:

\begin{equation}\begin{CD}
H_G(E, E-N;\eta) @>>j> H_G(E;\eta)\\
   @AA{\phi}A                   @VV{i^*}V \\
H_G(N; \eta)  @>>\cup Eul_G(E)> H_G(N;\eta)
\end{CD}\end{equation}
where $\phi$ is the Thom isomorphism of Proposition \ref{tocome}, and $j$ is induced by the inclusion map of forms.
\end{lemma}

\begin{proof}
The map $i^* \circ j \circ \phi$ is induced by a map of forms $h: \Omega_G(N) \rightarrow \Omega_G(N)$ defined by $h(\alpha) = i^*(exp(b)\wedge \pi^*( \alpha) \wedge \tau )$ where $\tau$ is a form representing the equivariant Thom class and $ b \in \Omega_G^2(N)$ satisfies $d_G(b) = \eta - \pi^* (i^* (\eta))$. We may choose $b$ so that $i^*(b) =0$, because if it doesn't we can replace it by $ b - \pi^*(i^*(b))$. Thus $h(\alpha) =  i^*(exp(b) \wedge \pi^*( \alpha) \wedge \tau)= i^*(exp(b)) \wedge \alpha \wedge i^*(\tau)  = \alpha \wedge i^*(\tau)$.  Because $i^*(\tau)$ represents the equivariant Euler class, this completes the proof.
\end{proof}

Hu and Uribe go on to prove the following twisted version of the localization theorem.

\begin{theorem}\label{local}(\cite{HuUribe06})
Let $T$ be a compact torus acting on a smooth, closed manifold $M$ and let $i: M^T \hookrightarrow M$ denote the inclusion of the fixed point set. Then for any $d_T$-closed 3-form $\eta \in \Omega^3_T(M)$, the kernel and cokernel of the induced map $i^*: H_T(M;\eta) \hookrightarrow H_T(M^T;\eta)$ are $\hat{S}\lie{t}^*$-torsion.
\end{theorem}



\subsection{Spectral Sequences}\label{Spectral Sequences}

We now consider two spectral sequences associated filtrations of the complex $(\hat{\Omega}_G(M), d_{\eta})$, both of which converge strongly to $H(M;\eta)$ (see Appendix \ref{apptwis} for an explanation of convergence properties).

First consider the filtration of $(\hat{\Omega}_G(M), d_{G,\eta})$

\begin{equation}\label{filtone}
F^p  = F^p \hat{\Omega}_G(M):= \prod_{k\geq p} \Omega_G^k(M)
\end{equation}
which satisfies $d_{G,\eta}(F^p) \subset F^{p+1}$. The resulting spectral sequence $(E_r^*, d_r)$ satisfies $E_1^p \cong E_2^p \cong H^p_G(M)$ the untwisted cohomology, while $d_2$ is the wedging map $ \eta \wedge (.) : H^*_G(M) \rightarrow H^*_G(M)$. Thus by Proposition \ref{carumba}, this spectral sequence collapses at $E_1$ if and only if $\eta$ is cohomologous to zero. In particular, if $\dim H(M) < \infty$ then in the nonequivariant case:

\begin{equation}\label{egadyikes}
\dim H(M ;\eta) \leq \dim H(M)
\end{equation}
with equality if and only if $\eta$ is $d$-exact.

Now consider a different filtration:

\begin{equation}\label{other}
L^p = L^p\hat{\Omega}_G(M) := \prod_{k \geq p} (\Omega(M) \otimes S^k\lie{g}^*)^G
\end{equation}
This gives rise to a spectral sequence $(E_r^*, d_r)$ of $(\hat{S}\lie{g}^*)^G$-modules satisfying $E_1^p \cong H(M;\eta(0)) \otimes (S^p \lie{g}^*)^G$, where $\eta(0)$ be the ordinary 3-form obtained by evaluating $\eta$ at $0\in \lie{g}$.

\begin{definition}
We say that a $G$-manifold $M$ is \textbf{$\eta$-equivariantly formal} if the spectral sequence defined above collapses at $E_1$. In this case $H_G(M;\eta)$ is noncanonically isomorphic to $H(M;\eta(0)) \otimes (\hat{S} \lie{g}^*)^G$ as a module over $(\hat{S} \lie{g}^*)^G$.
\end{definition}

Notice that the quotient complex $ \hat{\Omega}_G(M) / L^1 \hat{\Omega}_G(M)$ is canonically isomorphic to $\Omega(M)$. This gives rise to a natural map $H_G(M;\eta) \rightarrow H(M;\eta(0))$ for all twistings $\eta$. We have a version of the Leray-Hirsch theorem in this context.

\begin{proposition}\label{lerhirsh}
The $G$-manifold $M$ is $\eta$-equivariantly formal if and only if the natural map $ H_G(M ;\eta) \rightarrow H(M;\eta(0))$ is surjective.
\end{proposition}

\begin{proof}
The spectral sequence associated to the filtration collapses at page $E_1$ if and only if the injections $L^{p+1} \hookrightarrow L^p$ induce injections in cohomology $H( L^{p+1}\Omega_G(M) ; d_{G,\eta}) \hookrightarrow H( L^p \Omega_G(M) ; d_{G,\eta})$ for all $p$. By the associated six term exact sequence, is true if and only if $ H( L^p ; d_{G,\eta}) \rightarrow H(L^p /L^{p+1} ; d_{G,\eta})$ is surjective. Of course $H( L^p/L^{p+1}; d_{G,\eta}) = E_1^p \cong H(M ;\eta(0)) \otimes (S^p \lie{g}^*)^G$.

Collecting together, we see that $M$ is equivariantly formal if and only if the natural maps $$ \pi_p: H( L^p ; d_{G,\eta}) \rightarrow H(M ;\eta(0)) \otimes (S^p \lie{g}^*)^G$$ are surjective for all $p$. When $p=0$, $\pi_0$ is exactly $H_G(M;\eta) \rightarrow H(M;\eta(0))$ proving one direction of the equivalence. The opposite direction follows by noting that for $\sigma \in (S^p \lie{g}^*)^G$ and $d_{G,\eta}$-closed $\phi \in \Omega_G(M)$, we have $\pi_p( \phi \sigma) = \pi_0(\phi) \otimes \sigma$.
\end{proof}

\begin{proposition}\label{weylinv}
A $G$-manifold $M$ is $\eta$-equivariantly formal if and only if $M$ it is $\eta'$-equivariantly formal as a $T$-manifold under the restricted maximal torus action, where $\eta'$ is the image of $\eta$ under the induced map $\Omega_G^3(M) \rightarrow \Omega_T^3(M)$.
\end{proposition}

\begin{proof}
We use the criterion of Proposition \ref{lerhirsh}
The natural map $\phi: H_G(M ;\eta) \rightarrow H(M ;\eta(0))$ factors through the natural map $\phi'$ via $$ H_G(M;\eta) \rightarrow H_T(M;\eta') \stackrel{\phi'}{\rightarrow} H(M;\eta'(0)).$$ Thus if $\phi$ is surjective, so must $\phi'$.  On the other hand, the map $\phi'$ is $W$ equivariant, where the action of the Weyl group $W$ on $H(M;\eta(0))$ is induced by $N(T)$ action restricted from $G$. Since this action is isotropically trivial, we find that if $\phi'$ is invariant under the $W$ action, so is the restricted map $$ H_G(M;\eta) \cong H_T(M;\eta')^W \rightarrow H(M;\eta(0))$$completing the proof.
\end{proof}

It is worth noting that a $G$-manifold that is equivariantly formal for $\eta =0$ may fail to be formal for $\eta \neq 0$.

\begin{example}
Let $U(1)$ act trivially on $S^1$. Then $H_{U(1)}(S^1) \cong H(S^1) \otimes \R [x]$, where $0 \neq x \in \lie{u(1)}^*$. Choose a twisting $ \eta$ satisfying $0 \neq [\eta] \in H^1(S^1) \otimes x$. Then $H(S^1 ; \eta(0)) \cong H(S^1)$ because $[\eta(0)] = 0$, while $H_{U(1)}(M ;\eta) = 0$. The second assertion here follows from the $\{ L^p\}$ spectral sequence where $ E_1 \cong H(S^1) \otimes \R [x]$ with differential $d_1$ defined by wedging by $[\eta]$ so that $H(E_1, d_1) = E_2 = 0$.
\end{example}

\section{Morse theory}\label{morese theory}

Throughout this section, unless otherwise stated, $T$ is a compact torus with Lie algebra $\lie{t}$ acting on a closed smooth manifold $M$. Recall the following definition from \cite{GGK02}

\begin{definition} \label{absmom}
An nondegenerate abstract moment map $\mu : M \rightarrow \lie{t}^*$ is a smooth, equivariant map $\mu: M \rightarrow \lie{t}^*$ such that
for every vector $ \xi \in \lie{t}$,

(1) $Crit( \mu^{\xi} ) =
\{\xi_M = 0\}$
, and

(2) $\mu^{\xi} : M \rightarrow \R$ is a Morse-Bott function.
\end{definition}


Definition \ref{absmom} is an abstraction of the Morse theoretic properties of symplectic moment maps which are responsible for results such as Kirwan injectivity and surjectivity, as well as convexity (c.f. \cite{GGK02}). In \cite{NY07} Nitta actually proved that the components of moment map for Hamiltonian torus actions on compact generalized complex manifolds are abstract nondegenerate moment maps (see also \S \ref{g-moment-map}). Thus it follows that Kirwan injectivity and surjectivity for the usual equivariant cohomology must hold for GC-Hamiltonian actions.

To prove twisted versions of these theorems, we must impose a compatibility condition on the twisting 3-form.

\begin{definition}\label{compa}
Let $\eta \in \Omega_T^3(M) = \Omega^3(M)^T \oplus (\Omega^1(M)^T\otimes \lie{t}^*)$ and let $\eta^1$ denote the component of $\eta$ lying in $\Omega^1(M)^T\otimes \lie{t}^*$.  We say that $\eta$ is \textbf{compatible} if for all $p \in M$, $ker(\eta^1_p) \supset \lie{t}_p$, where $\eta^1$
is regarded as a linear map $\eta^1: \lie{t} \rightarrow \Omega^1(M)^T$ and $\lie{t}_p$ is the Lie algebra of the isotropy subgroup of the point $p\in M$.
\end{definition}

Our goal in this section is to prove the following two results:

\begin{theorem}\label{morseinj}[Kirwan Injectivity]
Let $M$ be a smooth, compact $T$-manifold with abstract moment map $\mu: M \rightarrow \lie{t}^*$, and compatible equivariantly closed 3-form $\eta \in \Omega_T^3(M)$. Then $M$ is \textbf{$\eta$-equivariantly formal}. In particular, the localization map $i^*: H_T(M;\eta) \rightarrow H_T(M^T; \eta)$ is injective and $$H_T(M ;\eta) \cong H(M;\eta(0))\otimes \hat{S}(\lie{t}^*)$$ noncanonically as $\hat{S}(\lie{t}^*)$-modules.
\end{theorem}

Theorem \ref{morseinj} may be generalized to noncompact $M$ using the weaker hypothesis that the fixed point set $M^T$ is compact and that some nonzero component of the moment map $\mu^{\xi} : M \rightarrow \R$ is proper and bounded below. Working in such generality is cumbersome, so we stick with compact $M$.

\begin{theorem}\label{morsesurj}[Kirwan Surjectivity]
Let $M$ be a smooth, compact $T$-manifold with abstract moment map $\mu: M \rightarrow \lie{t}^*$, and compatible equivariantly closed 3-form $\eta \in \Omega_T(M)$. Suppose $M$ admits an invariant almost complex structure. Then the map in equivariant cohomology induced by inclusion of the zero level set:

$$ H_T( M ; \eta) \rightarrow H_T( \mu^{-1}(0) ;\eta)$$
is a surjection. In the event that $T$ acts freely on $\mu^{-1}(0)$ then $H_T(\mu^{-1}(0);\eta) \cong H(\mu^{-1}(0)/T;\bar{\eta})$ as explained in Proposition \ref{freeaction}.
\end{theorem}

As before, Theorem \ref{morsesurj} may be generalized to include some examples of noncompact manifolds but for the sake of simplicity we work with compact $M$. In our proof of Theorem \ref{morsesurj} we found it necessary to require a invariant almost complex structure, though we suspect the theorem holds without this additional hypothesis. The presence of an invariant almost complex structure in the case of a GC Hamiltonian actions was proven by Nitta \cite{NY07} and played an important part in his work.


Let $f: M \rightarrow \R$ be a smooth function. We denote the critical set of $f$ by $Crit(f) = \{ x \in M | df_x = 0\}$. For $x \in Crit(f)$, the Hessian $Hess_x(f): T_xM \rightarrow T^*_xM$ is the symmetric linear map defined by the formula $$< Hess_x(f) (v) , w>= w \cdot L_{\tilde{v}}f,$$ where $v,w \in T_xM$, $\tilde{v}$ is any vector field satisfying $\tilde{v}_x = v$, $L$ is the Lie derivative and $<,>$ is the pairing between $T_xM$ and $T_x^*M$. The Hessian is more often defined as the quadratic form $<Hess_x(.),.>$, but the definition as a linear map is more convenient for us.

\begin{definition}
Let $M$ be a smooth, closed manifold. A smooth function $f:M \rightarrow \R$ is called \textbf{Morse-Bott} if the connected components of $Crit(f) = \{ x\in M |df_x =0\}$ are closed submanifolds of $M$ and for all $x \in Crit(f)$ the kernel of the Hessian satisfies $\ker(Hess_x(f)) = T_xCrit(f)$.
\end{definition}

Let $\{ C_i | i\in 0, 1,2,...,n\}$ be the set of connected components of $Crit(f)$. The function $f$ is constant on each component $C_i$ and we define $c_i := f(C_i) \in \R$. We will assume for simplicity of exposition that $c_i = c_j$ if and only if $i =j$, though all the proofs can be adapted to work without this assumption. We choose the indexing $i = 0,1,...,n$ so that $c_i < c_j$ if and only if $i <j$.

Choose a Riemannian metric $g$ on $M$. Using $g$ to identify $TM \cong T^*M$, we may regard $Hess_x(f)$ as an automorphism of $T_xM$ for $x\in Crit(f)$.  Because it is symmetric, $Hess_x(f)$ is diagonalizable with real eigenvalues. We define the negative normal bundle $\nu_i$ of $C_i$ by setting $\nu_{i, x}$ to equal the sum of negative eigenspaces of $Hess_x(f)$. Up to isomorphism, $\nu_i$ is independent of the choice of $g$. We call the rank of $\nu_i$ the \textbf{index} of $C_i$ and denote it $\lambda(i)$. In the presence of a compact torus $T$-action on $M$ leaving $f$ and $g$ invariant, the $\nu_i$ become equivariant vector bundles over $C_i$.

Let $M_t := f^{-1}((-\infty, t))$. If the interval $[s,t]$ contains no critical values for $f$, the inclusion $M_s \hookrightarrow M_t$ is a homotopy equivalence. In particular, if a torus $T$ acts on $M$ leaving $f$ invariant and $\eta \in \Omega_T^3(M)$ is a closed equivariant 3-form, then $H_T(M_t; \eta) \cong H_T(M_s;\eta)$.

Thus for some $\epsilon > 0 $ sufficiently small, we obtain for each critical value $c_i$ a six term exact sequences:

\begin{equation}\begin{CD}\label{rect}
H_T^0(M_{c_i+\epsilon}, M_{c_i-\epsilon}; \eta) @>>> H_T^0(M_{c_i+\epsilon};\eta) @>>> H_T^0(M_{c_i -\epsilon};\eta)  \\
@AAA  @.   @VVV\\
H_T^1(M_{c_i -\epsilon};\eta) @<<< H_T^1(M_{c_i+\epsilon};\eta) @<<< H_T^1(M_{c_i+\epsilon}, M_{c_i-\epsilon};\eta)
\end{CD}\end{equation}
and canonical isomorphisms $H_T(M_{c_i + \epsilon}) \cong H_T(M_{c_{i+1}-\epsilon})$.

Using excision and the Thom isomorphism, we obtain isomorphisms:

\begin{equation}\label{shortexit}
H_T^*(M_{c +\epsilon},M_{c-\epsilon}; \eta) \cong  H_T^*( \nu_i , \nu_i - 0; \eta) \cong  H_T^{*+\lambda(i)}(C_i; \eta)
\end{equation}
where the superscript grading is taken mod 2.





\begin{definition}
A $T$-invariant Morse-Bott function $f$ is called \textbf{$\eta$-equivariantly perfect} if the vertical arrows in (\ref{rect}) are zero for all critical values $c_i$.
\end{definition}

An important consequence is that $\oplus_i H_T(M_{c_i+\epsilon}, M_{c_i-\epsilon}; \eta) $ is isomorphic to an associated graded object of $H_T(M;\eta)$.

\begin{proposition}\label{smith}
Suppose that $f$ is bounded below, $\eta$-equivariantly perfect Morse-Bott function on $M$ and that the negative normal bundles are all orientable.  Then there is an isomorphism of $\hat{S} \lie{t}^*$-modules $$gr(H_T^*(M; \eta)) \cong \oplus_i H^{*+\lambda(i)}_T(C_i;\eta)$$ where $gr(H_T(M; \eta))$ is the associated graded ring determined by the topological filtration $M_s$ of $M$ and $\lambda(i) \in \{0,1\}$ is the index of $C_i$ mod $2$.
\end{proposition}

\begin{proof}
Because $f$ is $\eta$-equivariantly perfect, the exact sequence (\ref{rect}) decomposes into exact sequences \begin{equation} \label{exact-sequence} 0 \rightarrow H_T(M_{c_i+\epsilon}, M_{c_i-\epsilon}) \rightarrow H_T(M_{c_i+\epsilon};\eta)  \rightarrow H_T(M_{c_i -\epsilon};\eta) \rightarrow 0 \end{equation}
It follows that $gr( H_T(M; \eta)) \cong \oplus_i H_T(M_{c_i+\epsilon}, M_{c_i-\epsilon}) $. Applying (\ref{shortexit}) completes the proof.
\end{proof}

It was noticed by Atiyah and Bott that an invariant Morse-Bott function can sometimes be shown to be equivariantly perfect using only negative normal bundle data as follows. If the negative normal bundle is orientable, we construct a commutative diagram:

\begin{equation}\begin{CD}
 H_T(M_{c +\epsilon},M_{c-\epsilon}; \eta)@>j>>  H_T(M_{c +\epsilon};\eta) @>>> H_T(M_{c -\epsilon};\eta) \\
  @VV \cong V  @VVV\\            H_T( \nu_i , \nu_i - 0; \eta)  @>>> H_T( \nu_i; \eta)   \\
            @AA \cong A  @VV \cong V\\             H_T(C_i; \eta)    @>>\cup \text{Eul}_T(\nu_i)>  H_T(C_i; \eta)
\end{CD}\end{equation}
where the upper square is excision and the bottom square is from Lemma \ref{euler}. If $\cup Eul_T(\nu_i): H_T(C_i; \eta) \rightarrow H_T(C_i; \eta)$ is injective then $j$ must also be injective. We obtain the self perfecting principle:

\begin{lemma}\label{astral}
Suppose that for all critical sets $C_i$, $Eul_T(\nu_i)$ is not a zero divisor for $H_T(C_i;\eta)$, i.e. for all $\alpha\in H_T(C_i; \eta)$, we have $ \alpha \cup Eul_T(\nu_i) = 0$ if and only if $\alpha = 0$. Then $f$ is $\eta$-equivariantly perfect.
\end{lemma}

Atiyah and Bott discovered a simple criterion implying that $Eul_T(\nu_i)$ is not a zero divisor in the nontwisted setting. We adapt their proof to the twisted case.

\begin{lemma}\label{atbot}
Let $\nu \rightarrow N$ be a $T$-equivariant oriented vector bundle over a compact manifold $N$ and suppose there exists a subtorus $ S \subset T$ such that $\nu^S$ is exactly the zero section of $\nu$. Let $\eta \in \Omega_{T/S}(N) \hookrightarrow \Omega_T(N)$ under the natural inclusion (see Example \ref{trivialact2}).  Then $Eul_T(\nu)$ is not a zero divisor for $H_T(N;\eta)$.
\end{lemma}

\begin{proof}

By Example \ref{trivialact2} the untwisted equivariant cohomology satisfies$$H_T(N) \cong H_{T/S}(N) \otimes \hat{S}(\lie{s}^*).$$ It was shown in \cite{AB82} \S 13, that the equivariant Euler class $Eul_T(\nu) \in H_T(N)$ satisfies

$$ Eul_T(\nu) = 1 \otimes \beta_0 + \text{positive degree terms in } H^*_{T/S}(N)$$
where $\beta_0 \in S(\lie{s}^*)$ is nonzero.

Also by Example \ref{trivialact2}, we have a natural isomorphism $$H_T(N; \eta) \cong H_{T/S}(N;\eta) \otimes \hat{S}(\lie{s}^*)$$ The ideal $I = \prod_{k>0}H_{T/S}^{k}(N)$ is the Jacobson ideal of $H_{T/S}(N)$, so by Nakayama's Lemma the filtration $ \{F^p\}$ of $H_T(N;\eta)$ defined by $F^p := I^p \cup H_{T/S}(N;\eta) \otimes \hat{S}(\lie{s}^*)$ satisfies $ \cap_p F^p =0$ (see Appendix \ref{commalg}). For $\alpha \in H_T(N;\alpha)$ nonzero, define $p(\alpha)$ by $ \alpha \in F^{p(\alpha)} - F^{p(\alpha)+1}$. It follows that $$ \alpha \cup Eul_T(\nu) = \alpha \cup \beta_0  \text{ modulo } F^{p+1}$$ which is nonzero.
\end{proof}

Notice that Definition \ref{compa} ensures that if $N$ is a component of $M^S$ and $\nu$ is a subbundle of the normal bundle of $N$ in $M$, then the hypotheses of Lemma \ref{atbot} apply.

\begin{lemma}\label{patti}
Under the hypotheses of Theorem \ref{morseinj}, for a generic choice of $\xi \in \lie{t}$, the moment map component $\mu^{\xi}:M \rightarrow \R$ is $\eta$-equivariantly perfect.
\end{lemma}

\begin{proof}

For generic choice $\xi \in \lie{t}$, the image of $exp: \text{span} \{\xi\} \rightarrow T$ is dense. Letting $f = \mu^{\xi}$, it follows that $$ Crit(f) = \{p \in M| \xi_M = 0\} = M^T$$
By Lemma \ref{astral}, it suffices to show for each connected component $C_i$ of $M^T$ with negative normal bundle $\nu_i \rightarrow C_i$, that $H(C_i; \eta)$ possesses no $Eul_T(\nu_i)$-torsion. This is a consequence of Lemma \ref{atbot} in the case $S=T$.
\end{proof}

\begin{proof}[Proof of Theorem \ref{morseinj}]
By Proposition \ref{lerhirsh}, equivariant formality is equivalent to surjectivity of the natural map $H_T(M;\eta) \rightarrow H(M;\eta)$. We prove this by induction on $H(M_t;\eta)$ where $M_t := f^{-1}((-\infty, t))$ and $f$ a generic component of the moment map as in Lemma \ref{patti}.

For the base case $M_t$ is empty for small $t$ because $M$ is compact.

In the induction step, assume that $H_T( M_{c_i - \epsilon};\eta) \rightarrow  H( M_{c_i - \epsilon};\eta)$ is surjective. Using long exact sequences for the pair and Lemma \ref{patti} we obtain a commutative diagram:

\begin{equation}\begin{CD}
\xymatrix{  0 \ar[r] & H_T(  M_{c_i + \epsilon}, M_{c_i - \epsilon};\eta)\ar[d] \ar[r]& H_T( M_{c_i + \epsilon};\eta) \ar[d] \ar[r] &H_T( M_{c_i - \epsilon};\eta) \ar[r] \ar @{->>}[d]&0\\
& H(  M_{c_i + \epsilon}, M_{c_i - \epsilon};\eta) \ar[r]& H( M_{c_i + \epsilon};\eta) \ar[r] &H( M_{c_i - \epsilon};\eta)  &
}
\end{CD}\end{equation}
By a diagram chase we are reduced to proving that $H_T(  M_{c_i + \epsilon}, M_{c_i - \epsilon};\eta) \rightarrow H(  M_{c_i + \epsilon}, M_{c_i - \epsilon};\eta)$ surjects. By the Thom isomorphism this is equivalent to showing that the critical sets $C_i$ are equivariantly formal. By the compatibility of $\eta$ this follows from Example \ref{trivialact}.

The injectivity of $H_T(M;\eta) \rightarrow H_T(M^T;\eta)$ follows from \ref{local}, because $H_T(M;\eta)$ is torsion free.
\end{proof}

\begin{corollary}\label{qudimes}
Under the hypotheses of \ref{morseinj} we have an equality $$ \dim H(M; \eta(0)) = \dim H(M^T;\eta(0))$$
\end{corollary}

\begin{proof}
By equivariant formality $$ H_T(M ;\eta) \cong H(M;\eta(0)) \otimes \hat{S} \lie{t}^*$$ while by equivariant perfection of a generic component of the Morse map $$ gr ( H_T(M ;\eta)) \cong \oplus H(C_i ; \eta(0)) \otimes \hat{S} \lie{t}^*$$ if we ignore the $\Z_2$-grading. The summands $H(C_i ; \eta(0))  \otimes \hat{S} \lie{t}^*$ are free, hence projective over $\hat{S} \lie{t}^*$ so $$ H_T(M;\eta) \cong H( \cup_i C_i ; \eta(0)) \otimes \hat{S} \lie{t}^* = H( M^T; \eta(0))\otimes \hat{S} \lie{t}^*.$$
\end{proof}

We now turn our attention to the proof of the Kirwan surjectivity Theorem \ref{morsesurj}.  We will need the following proposition.

\begin{proposition}\label{immunology}
Let $M$ be a compact $T$-manifold with nondegenerate abstract moment map $\mu: M \rightarrow \lie{t}^*$ for which $0 \in \lie{t}^*$ is a regular value, and suppose that $M$ admits a $T$-invariant almost complex structure. Then we may choose a basis $\xi_1, ..., \xi_n$ of $\lie{t}$ such that

(1) each $\lie{t}_k := Span\{\xi_1,...,\xi_k\}$ exponentiates to a rank $k$ torus $T_k$,

(2) $0 \in \lie{t}_k^*$ is a regular value for the moment map $\mu_k = proj_{\lie{t}_k^*} \circ \mu$

(3) The restriction of $\mu^{\xi_{k+1}}$ to the submanifold $M_k = \mu_k^{-1}(0) \subset M$ is Morse-Bott with critical set equal to the points where $T_{k+1}$ acts with positive dimensional stabilizer.
\end{proposition}

The proof of Proposition \ref{immunology} is postponed until Appendix \ref{nondegabs}.

Proposition \ref{immunology} allows us to factor the Kirwan map $H_T(M;\eta) \rightarrow H_T(\mu^{-1}(0))$ through the sequence of submanifolds determined by Proposition \ref{immunology} $$H_T(M) \rightarrow H_T(\mu^{-1}_1(0)) \rightarrow H_T(\mu^{-1}_2(0)) ...\rightarrow H_T(\mu^{-1}_n(0)) = H_T(\mu^{-1}(0)).$$ Our strategy to prove Theorem \ref{morsesurj} is to show that each map in this composition is surjective. We do this by applying the following lemma to the $T$-manifold $\mu^{-1}_k(0)$ with function $(\mu^{\xi_{k+1}})^2$, which completes the proof of Theorem \ref{morsesurj}.

\begin{lemma}\label{surj}
Let $X$ be compact smooth $T$-manifold with no orbits of dimension smaller than $d$, and let $N$ be the union of dimension $d$ orbits. Let $f: X \rightarrow \R$ be an $T$-invariant function such that $Crit(f) = N \cup f^{-1}(c_0)$, where $c_0$ is the minimum value of $f$. Suppose that $f$ is Morse-Bott except possibly at the minimum $f^{-1}(0)$. For $C=C_i$ a connected component of $N$, let $\lie{t}_C$ denote the infinitesimal stabilizer of $C$ and let $T_C = exp(\lie{t}_C)$ its torus. If  $\eta \in \Omega^3_{T}(X)$ is a $d_{T}$-closed form satisfying $$ \eta|_{C} \in \Omega_{T/T_C}(C)\otimes 1 \subset  \Omega_{T/T_C}(C)\otimes S \lie{t}_C^* \cong \Omega_{T}(C) ,$$ then the map induced by inclusion $$ H_{T}(X;\eta) \rightarrow H_{T}(f^{-1}(c_0);\eta)$$ is surjective.
\end{lemma}

\begin{proof}
The map $H_T( X_{c_0 + \epsilon}; \eta) \rightarrow H_T( f^{-1}(c_0);\eta)$ is an isomorphism, hence surjective.

Now suppose inductively that $H_T(X_{c_{i}-\epsilon};\eta) \cong H_T( X_{c_{(i-1)} + \epsilon};\eta) \rightarrow H_T( f^{-1}(c_0);\eta)$ is surjective for some $i$. We must show that $H_T(X_{c_{i}+\epsilon};\eta) \rightarrow H_T(X_{c_{i}-\epsilon};\eta)$ is surjective.
By Lemma \ref{astral}, it will suffice to show that $Eul_T(\nu_i)$ is not a zero divisor for $H_T( C_i ; \eta)$, where $C_i$ is a connected component of $N$. This follows from Lemma \ref{atbot} using $S = T_{C_i}$.
\end{proof}

\begin{proof}[Proof of Theorem \ref{morsesurj}]

\end{proof}

\section{Generalized complex geometry}\label{g.c.geometry.review}
Let $V$ be an $n$ dimensional vector space. There is a natural bi-linear pairing
of signature $(n, n)$ on $V\oplus V^*$ which is defined by
                            \[    \langle X +\alpha , Y +  \beta \rangle = \dfrac{1}{2}(\beta (X) +
                            \alpha(Y)).\]
A \textbf{generalized complex structure} on a vector space $V$ is an
orthogonal linear map $\mathcal{J}: V\oplus V^* \rightarrow V\oplus
V^*$ such that $\mathcal{J}^2=-1$. Let $L \subset V_{\C}\oplus
V^*_{\C}$ be the $\sqrt{-1}$ eigenspace of the generalized complex
structure $\mathcal{J}$. Then $L$ is maximal isotropic and $L \cap
\overline{L} = \{0\}$. Conversely, given a maximal isotropic
$L\subset V_{\C}\oplus V^*_{\C}$  so that  $L \cap \overline{L} =
\{0\}$, there exists an unique generalized complex structure
$\mathcal{J}$ whose $\sqrt{-1}$ eigenspace is exactly $L$.




Let $M$ be a manifold. A \textbf{generalized almost complex structure} on a manifold $M$ is
an orthogonal bundle map $\mathcal{J}:TM\oplus T^*M \rightarrow
TM\oplus T^*M$ such that for any $x\in M$, $\mathcal{J}_x$ is a generalized complex structure on the vector space $T_xM$.

Given a closed three form $H \in \Omega^3(M)$, an $H$-\textbf{twisted generalized complex structure} $\J$  is
a generalized almost complex structure such that the sections of the $\sqrt{-1}$ eigenbundle of $\J$ are closed
under the $\eta$-\textbf{twisted Courant
bracket}, i.e., the bracket defined by the formula
\[
[X+\xi,Y+\zeta]=[X,Y]+L_X\zeta-L_Y\xi-\dfrac{1}{2}d\left(\zeta(X)-\xi(Y)\right)+\iota_Y\iota_X H.\]









A generalized almost K\"ahler structure is a pair of two
commuting generalized almost complex structures $\J_1$,$\J_2$
such that $\langle -\J_1\J_2\xi,\xi\rangle
>0$ for any $\xi \neq 0 \in C^{\infty}(T_{\C}M \oplus  T^*_{\C}M)$,
where $\langle\cdot,\cdot \rangle$ is the canonical pairing on
$T_{\C}M \oplus  T^*_{\C}M$. A generalized almost K\"ahler structure
$(\J_1,\J_2)$ is called an $H$-twisted generalized K\"ahler structure if both $\J_1$ and $\J_2$ are $H$-twisted generalized complex structures.
Given a generalized almost K\"ahler structure $(\J_1,\J_2)$, define
$\mathcal{G}(A,B):=\langle -\J_1\J_2A,B \rangle$, $A,B \in C^{\infty}(TM\oplus T^*M)$. Then $\mathcal{G}$ is a Riemannian metric on
$TM\oplus T^*M$, and its restriction to $TM$ defines a Riemannian metric $g$ on $M$.
Let $G=-\J_1\J_2$. Since $G^2=id$, $TM\oplus T^*M=C_+\oplus C_-$, where $C_{\pm}$ is the $\pm 1$-eigen-bundle of $G$.
Let $\pi: TM\oplus T^*M \rightarrow TM$ be the projection map. Then \[ \pi\mid_{C_{\pm}}: C_{\pm} \rightarrow TM \] is an isomorphism.
Since $\J_1$ commutes with $G$, $C_{\pm}$ is invariant under $\J_1$. By projecting from $C_{\pm}$, $\J_1$ induces two almost complex structure
$I_+$ and $I_-$ on $TM$ which are compatible with the Riemannian metric $g$.

Conversely, if there are two almost complex structures $I_+$ and $I_-$ which are compatible with a Riemannian metric $g$ on $M$, then
\begin{equation} \label{g.k.pair} \J_{1/2}=\dfrac{1}{2}\left(\begin{matrix} &1 & 0\\ & b &
1\end{matrix}\right) \left(\begin{matrix}& I_+\pm I_- &
-(\omega_+^{-1}\mp \omega_-^{-1})\\ &\omega_+\mp \omega_- &
-(I_+^*\pm I_-^*)\end{matrix} \right)\left(\begin{matrix} &1 & 0\\
&- b & 1\end{matrix}\right)
 \end{equation}
is a generalized almost K\"ahler structure, where $\omega_{\pm}=gI_{\pm}$ are the fundamental $2$-forms of the
Hermitian structures $(g,I_{\pm})$, and $b$ is a two form. The following remarkable result is due to Gualtieri.

\begin{theorem} \label{g.k.approach-to-bi-hermitian} (\cite{Gua03}) An $H$-twisted \textbf{generalized K\"ahler
structure} is equivalent to a triple $(g,I_+,I_-)$ consisting of a
Riemannian metric $g$ and two integrable almost complex structure
compatible with $g$, satisfying the integrability conditions:
\[d_+^c\omega_++d^c_-\omega_-=0, \,\, H+db=d_+^c\omega_+,\,\, dd^c_{\pm}\omega_{\pm}=0,\]
where $\omega_{\pm}=gI_{\pm}$, $d^c_{\pm}$ are the
$i(\overline{\partial}-\partial)$ operator associated to the complex
structure $I_{\pm}$, and $b$ is a two form. In particular, a triple
$(g,I_+,I_-)$ satisfying the above assumption defines a generalized
K\"ahler pair $\J_1, \J_2$ by the formula (\ref{g.k.pair}).

 \end{theorem}

We close this section with a quick review of generalized complex submanifolds as introduced in \cite{BB03} (See also \cite{BS06}). Although \cite{BB03} only defined
generalized complex submanifolds for untwisted generalized complex structures, the definition given there extends naturally to the twisted case as well.

Let $W$ be a submanifold of an $\eta$-twisted generalized complex manifold $(M, \J)$, let $L\subset T_{\C}M\oplus T^*_{\C}M$ be the $\sqrt{-1}$-eigenbundle of $\J$, and let  $i:W \rightarrow M$ be the inclusion map. At each point $x \in N$ set

\[L_{W,x}=\{X+(\xi\mid_{T_{\C}W}) : X+\xi \in L\cap (T_{\C,x}W\oplus T^*_{\C,x}M)   \} . \]

This defines a maximally isotropic distribution of $T_{\C}W\oplus T_{\C}^*W$ whose sections are closed under the $i^*H$-twisted Courant bracket.
If $L_W$ is a sub-bundle of $T_{\C}W\oplus T^*_{\C}W$ and if $L_W \cap \overline{L}_W=0$, then $W$ is said to be a generalized complex submanifold. \footnote{ It is noteworthy that the sufficient and necessary conditions for $W$ to be a generalized complex submanifolds have been found in \cite{BS06}.} It is clear from the definition that if $W$ is a generalized complex submanifold then there exists a unique $i^*H$-twisted generalized complex structure $\J_W$ on $W$ whose $\sqrt{-1}$-eigenbundle is exactly $L_W$.

It is well-known that the fixed point submanifold of a symplectic torus action on a symplectic manifold is a symplectic submanifold.
\cite{Lin06} extends this fact to generalized complex manifolds.

\begin{proposition}\label{fixed-points-submfold} Suppose the action of a torus $T$ on an $H$-twisted generalized complex manifold $(M,\J)$ preserves the generalized complex structure $\J$. And suppose $Z$ is a connected component of the fixed point set. Then $Z$ is a generalized complex submanifold of $M$. Let $i: Z\rightarrow M$ be the inclusion map. Then $Z$ carries a $i^*H$-twisted generalized complex structure. \end{proposition}

 \section{Generalized moment maps}\label{g-moment-map}
First we recall the definition of Hamiltonian actions on
$H$-twisted generalized complex manifolds given in \cite{LT05}.

\begin{definition}\footnote{Indeed, Condition (b) was not imposed
in \cite[Definition A.2.]{LT05}. However, in order to make the quotient construction work,
Tolman and the author made it clear in \cite[Prop. A.7, A.10]{LT05} that $H+\alpha$ must be equivariantly closed in the usual equivariant
Cartan model.}
\cite{LT05}) \label{deftmm}
Let a compact Lie group $G$ with Lie algebra $\g$ act on a manifold
$M$, preserving an $H$-twisted generalized complex structure
$\mathcal{J}$, where $H \in \Omega^3(M)^G$ is closed.
The action of $G$ is said to be Hamiltonian if there exists a smooth
equivariant function $\mu:M \rightarrow \frak{g}^*$,  called the
{\bf generalized moment map}, and a 1-form $\alpha \in
\Omega^1(M,\g^*)$, called the {\bf moment one form},  so that
\begin{itemize}
\item [a)]
$\J d\mu^{\xi}=-\xi_M-\alpha^{\xi}$ for all $\xi \in \g$, where $\xi_M$ denotes the
induced vector field.
\item [b)] $H+\alpha$ is an equivariantly closed three form in the
usual Cartan Model.
\end{itemize}
\end{definition}

\begin{remark} \label{poisson-Hamiltonian}An $H$-twisted generalized complex structure $\J: TM\oplus T^*M \rightarrow TM\oplus T^*M$ induces by restriction and
projection a map $\beta: T^*M \rightarrow TM$ which is a real Poisson bi-vector, see for instance \cite{Gua03} and \cite{BS06}. If the action of a
compact Lie group $G$ on a generalized complex manifold $(M, \J)$ is Hamiltonian with a generalized moment map $\mu \rightarrow \frak{g}^*$, then a straightforward calculation shows
\[ -\beta (d\mu^{\xi}) = \xi_M,\] where $\xi_M$ is the vector field on $M$ induced by $\xi \in \frak{g}$. This shows clearly that the action of $G$ is Hamiltonian with respect to the Poisson bi-vector $\beta$.

\end{remark}

Let a compact Lie group $G$ act on a twisted generalized complex
manifold $(M,\J)$ with generalized moment map $\mu$. Let $\cO_a$ be
the co-adjoint orbit through $a \in \g^*$. If $G$ acts freely on
$\mu^{-1}(\cO_a)$, then $\cO_a$ consists of regular values and $M_a
= \mu^{-1}(\cO_a)/G$ is a manifold, which is called the {\bf
generalized complex quotient}. The following  two results were
proved in \cite{LT05}.

\begin{lemma}\label{twist three form}
Let a compact Lie group $G$ act freely on a manifold $M$. Let $H$ be
an invariant closed three form and let $\alpha$ be an equivariant
mapping from $\g$ to $\Omega^1(M)$.  Fix a connection $\theta \in
\Omega(M,\g^*)$. Then if $H + \alpha \in \Omega^3_G(M)$ is
equivariantly closed, there exists a natural form $\Gamma \in
\Omega^2(M)^G$ so that $\iota_{\xi_M} \Gamma = \alpha^\xi$. Thus $H
+ \alpha + d_G \Gamma \in \Omega^3(M)^G \subset \Omega_G^3(M)$ is
closed and basic and so descends to a closed form $\widetilde{H} \in
\Omega^3(M/G)$ so that $[\widetilde{H}]$ is the image of $[H  +
\alpha]$ under the Kirwan map.
\end{lemma}

\begin{proposition} \label{Twisted Complex Reduction}
Assume there is a Hamiltonian action of a compact Lie group $G$ on
an $H$-twisted generalized complex manifold $(M,\mathcal{J})$ with
generalized moment map $\mu: M \to \g^*$ and moment one-form $\alpha
\in \Omega^1(M,\g^*)$. Let $\cO_a$ be a co-adjoint orbit through $a
\in \g^*$ so that $G$ acts freely on $\mu^{-1}(\cO_a)$. Given a
connection on $\mu^{-1}(\cO_a)$, the generalized complex quotient
$M_a$ inherits an $\widetilde{H}$-twisted generalized complex
structure $\widetilde{\J}$, where $\widetilde{H}$ is defined as in
the Lemma \ref{twist three form}. Up to $B$-transform,
$\widetilde{J}$ is independent of the choice of connection.
\end{proposition}

\subsection{Nitta's theorem and compatibility}

The rest of this section is devoted to the proof of Proposition \ref{moment-map-Morse-Bott}, which indeed has already been established by Nitta in \cite{NY07}. However, since Proposition \ref{moment-map-Morse-Bott} is central to our paper and since we believe more details are needed in Nitta's argument to make it more accessible, we will present a self-contained detailed proof in our paper. The essential step in the proof of Proposition \ref{moment-map-Morse-Bott} is the non-trivial observation that the restriction of $\alpha^{\xi}$ to the fixed point set $F^{\xi}$ vanishes. To prove it in full generality, we are going to use the maximum principle of pseudo-holomorphic functions on almost complex manifolds, as advocated in \cite[Prop. 3.1]{NY07}. Note that our proof differs slightly from the one given in \cite{NY07}. For instance, our proof does not involve the use of a
Levi-Civita connection, and we apply Proposition \ref{fixed-points-submfold} in an essential way. We would also like to mention that when the
generalized complex manifolds have constant types, one can construct more elementary proofs using the Darboux theorem of generalized complex structures \cite{Gua03}.

\begin{lemma}\label{trivial-Ham-action} Suppose the trivial action of a torus $T$ on a compact $H$-twisted generalized complex manifold $(M,\J)$ is Hamiltonian with a generalized moment map $\mu$ and a moment one form $\alpha$. Then $d\mu^{\xi}=\alpha^{\xi}=0$ for all $\xi \in \frak{t}$.

\end{lemma}

\begin{proof}  It has been shown that there exists a generalized almost complex structure $\J_2$ such that $\J_1=\J$ and $\J_2$ form a generalized almost K\"ahler pair, see for instance, \cite[Sec. 3]{Ca06} and \cite{NY07}. As we explained in Section \ref{g.c.geometry.review}, the generalized almost K\"ahler structure induces a triple $(g,I_+,I_-)$ consisting of a Riemannian metric $g$ and two almost complex structures $I_+$ and $I_-$ compatible with $g$; moreover, one can reconstruct $\J_1$ and $\J_2$ from
 the triple $(g,I_+,I_-)$ using Formula (\ref{g.k.pair}). Given $\xi \in \frak{t}$, by assumption  we have
$\J_1 d\mu^{\xi} = \alpha^{\xi}$, i.e.,
\begin{equation} \label{ham-eq} \dfrac{1}{2}\left(\begin{matrix} &1 & 0\\ & b &
1\end{matrix}\right) \left(\begin{matrix}& I_+ + I_- &
-(\omega_+^{-1}- \omega_-^{-1})\\ &\omega_+-\omega_- &
-(I_+^*+ I_-^*)\end{matrix} \right)\left(\begin{matrix} &1 & 0\\
&- b & 1\end{matrix}\right)\left(\begin{matrix}& 0\\ & d\mu^{\xi} \end{matrix}\right)=\left(\begin{matrix}& 0\\ & \alpha^{\xi} \end{matrix}\right), \end{equation}
where $\omega_{\pm}=gI_{\pm}$. A straightforward calculation shows that
\[ \omega_+^{-1}(d\mu^{\xi})=\omega_-^{-1}(d\mu^{\xi}),\,\,\, I_+^*d\mu+I_-^*d\mu^{\xi}=2\alpha^{\xi}.\]

Since $I_{\pm}$ are compatible with $g$, we have $\omega_{\pm}=gI_{\pm}=-I_{\pm}^*g$ and so $\omega_{\pm}^{-1}=g^{-1}I_{\pm}^*$.  It follows from $\omega_+^{-1}(d\mu^{\xi})=\omega_-^{-1}(d\mu^{\xi})$ that $I_+^*d\mu^{\xi}=I_-^*d\mu^{\xi}$. Thus $I^{*}_{\pm}d\mu^{\xi}=\alpha^{\xi}$. However, since by assumption the trivial action is Hamiltonian, condition b) in Definition \ref{deftmm} implies that $d\alpha^{\xi}=0$. Locally, we can always find a function $h$ such that $I^*_{\pm}d\mu^{\xi}=dh$.
If the generalized almost complex structure $\J_2$ is integrable, i.e., $(\J_1,\J_2)$ forms a generalized K\"ahler pair, then
$I_{\pm}$ must be integrable complex structures and $\mu^{\xi}$ is locally the real part of the $I_{\pm}$-holomorphic function $\mu^{\xi} + \sqrt{-1}h$.  If $M$ is compact, it follows from the maximum principle of the real part of a holomorphic function that $\mu^{\xi}$ has to be a constant. In the general case, locally $\mu^{\xi}$ is the real part of a pseudo-holomorphic function with respect to almost complex structures $I_{\pm}$. By the maximum principle of the real part of a pseudo-holomorphic function as we explained in Appendix \ref{maximum-principle}, $\mu^{\xi}$ has to be a constant provided $M$ is compact.

\end{proof}

Now consider the Hamiltonian action of a compact connected torus $T$ on a compact $H$-twisted generalized complex manifold $M$ with a generalized moment map $\mu$. Suppose that $\xi \in \t$ generates a compact connected sub-torus $T_1$ in $T$, i.e., $\xi$ is a generic element in the Lie algebra $\frak{t}_1$ of $T_1 \subset T$.

\begin{lemma} \label{crticial-set-equal-fixed-point-set}  Under the above assumptions, the critical set \[ \text{Crit}(\mu^{\xi})=\{ x\in M \mid (d\mu^{\xi})_x=0\} \] coincides with the fixed point set $F$ of the $T_1$ action on $M$ for any $\xi \in \frak{g}$.
\end{lemma}

\begin{proof} The inclusion $ \text{Crit}(\mu^{\xi}) \subset F$ is obvious. It suffices to show that for any
 $x \in F$ we have $(d\mu^{\xi})_x=0$. By Proposition \ref{fixed-points-submfold} the fixed point set $F$ of the $T_1$-action is a generalized complex submanifold. Moreover, it follows from \cite[Lemma 4.8]{Lin07} that the induced trivial action of $T_1$ on $F$ is Hamiltonian with the generalized moment map $\mu \mid_{F}: F\rightarrow \frak{t}_1^*$. Since $M$ is compact, $F$ has to be compact itself. By Lemma \ref{trivial-Ham-action}, we have $(d\mu^{\xi})\mid_{F}=0$. Choose a $T_1$-invariant metric on $M$. Since $\mu^{\xi}$ is $T_1$-invariant, $\text{grad}(\mu^{\xi})$, the gradient flow of $\mu^{\xi}$, is also invariant under the linearized $T_1$ action. It follows that at each point $x \in F$, $\text{grad}(\mu^{\xi})$ is tangent to $F$. Thus $
\langle \text{grad}\mu^{\xi}, d\mu^{\xi}\rangle_x=0$. This implies that for any $x \in F$,  $(d\mu^{\xi})_x=0$.

\end{proof}

\begin{remark} It is clear from the proof that in the statement of Lemma \ref{crticial-set-equal-fixed-point-set}, we need only to assume that all fixed points submanifolds are compact. This compactness assumption here is essential. For instance, Hu \cite[Sec. 4.5]{Hu05} constructed
an example of a Hamiltonian $S^1$ action on a generalized complex manifold with a proper moment map $f$ such that $\crit (f) \nsubseteqq M^{S^1}$.
Lemma \ref{crticial-set-equal-fixed-point-set} fails because in Hu's example the fixed point submanifold is a copy of complex plane $\C$ which is non-compact.
\end{remark}


  We are ready to give a proof of Proposition \ref{moment-map-Morse-Bott}. We are going to use the same notations as in Lemma \ref{crticial-set-equal-fixed-point-set}.
\begin{proof}[Proof of Proposition \ref{moment-map-Morse-Bott}]
Let $f := \mu^{\xi}$. In view of Lemma \ref{crticial-set-equal-fixed-point-set}, it suffices to show that the Hessian of $f$ is nondegenerate in the normal direction to $M^{T_1}$.

Because $\xi$ is generic, $M^{T_1} = ker(X)$. The vector field $X$ linearizes at $p \in M^{T_1}$ to $\mathcal{A} \in End(T_pM)$ defined by the formula $\mathcal{A}(w_p) = [X,w]_p$ for $w \in C^{\infty}(TM)$  . Since $T_1$ is connected we have $T_p(M^{T_1}) = \ker \mathcal{A}$ as subsets of $T_pM$.
The Hessian of $f$ at $p$ is a linear map  $Hess_p(f): T_pM \rightarrow T_p^*M$, defined by $ Hess_p(f)(w_p) = (d L_w (f))_p$ where $w\in C^{\infty}(TM)$. We need to show that for all $p \in \crit(f)$, $\ker(Hess_p(f)) \subset T_p \crit(f)$.

Now let $\beta$ be the canonical Poisson bivector associated to the generalized complex structure $\J$. As we explained in Remark \ref{poisson-Hamiltonian}, we have $X = -\beta df$. Thus
\begin{align*}
-[X,w] = L_w X =- L_w (\beta df) =- (L_w \beta)df- \beta( L_w df).
\end{align*}
Thus,
\begin{align*}
\mathcal{A}(w_p) = (L_w \beta)_pdf_p + \beta_p( L_w df)_p = \beta_p( L_w df)_p = \beta_p( Hess_p(f)(w_p)),
\end{align*}
where we've used that $df_p = 0$ and $( L_w df)_p = (d L_w f)_p = Hess_p(f)(w_p)$. Thus $\ker(Hess_p(f)) \subset \ker(\mathcal{A}) = T_p(M^{T_1}) = T_p \crit(f)$ as desired.
\end{proof}

\begin{remark} Choose an invariant generalized almost complex structure $\J_2$ such that $(\J,\J_2)$ form a generalized almost K\"ahler pair. It is easy to show that for any $p\in M$, $T_pM$ splits as the direct sum of $T_pM^{T_1}$ and $N$, where $N$ is the orthogonal complement of $T_p M^{T_1}$ in $T_pM$ with respect to the Riemannian metric induced by the generalized almost K\"ahler pair $(\J,\J_2)$; moreover, the vector space $N$ inherits a generalized K\"ahler structure which is invariant under the linearized action of $T_1$ on $N$.  It follows that $N$  admits a complex structure which is invariant under the operator $\mathcal{A}$ which we defined in the proof of Proposition \ref{moment-map-Morse-Bott}. As a direct consequence, we see that the $Hess_p(f)$ must have even index. So $f=\mu^{\xi}$ must be a Morse-Bott function of even index.
\end{remark}

It follows easily from Proposition \ref{moment-map-Morse-Bott} that the twisting form $H+\alpha$ is compatible with the torus action (Definition \ref{compa}).

\begin{corollary}\label{compmom}
Let $T \times M \rightarrow M$ be a Hamiltonian $T$-action for a compact, connected $H$-twisted generalized complex manifold $M$ with moment map $\mu: M \rightarrow \lie{t}^*$ and moment 1-form $\alpha$. For $x \in M$, denote $\lie{t}_x$ to be the infinitesimal stabilizer of $x$. Then $$ ker( \alpha_x) \supset \lie{t}_x$$ where we regard $\alpha_x$ as an element of $Hom( \lie{t}, T_x^*M)$.
\end{corollary}

\begin{proof}

Condition a) of Definition \ref{deftmm} asserts

$$ -\xi_M - \alpha^{\xi} =  \mathcal{J} d \mu^{\xi}$$
for all $\xi \in \lie{t}$. Proposition \ref{moment-map-Morse-Bott} says that if $(\xi_M)_x = 0$ then $d \mu^{\xi}_x =0$ and consequently $\alpha_x^{\xi} =0$.
\end{proof}

\section{Kirwan injectivity and Surjectivity and its application} \label{main-results}

 Proposition \ref{moment-map-Morse-Bott} establishes that the moment map $\mu$ of a compact $H$-twisted generalized Hamiltonian $T$-space is a nondegenerate abstract moment map and Corollary \ref{compmom} establishes that the equivariant twisting 3-form $H+\alpha$ is compatible. Thus Theorems \ref{kirinj} and \ref{kirsur} follow from Theorem \ref{morseinj} and Theorem \ref{morsesurj} respectively. Theorem \ref{eqformality} then follows from Proposition \ref{weylinv} because the restriction of a generalized Hamiltonian action is generalized Hamiltonian.


One of the early motivations for this paper was to prove the following result.

\begin{theorem}\label{applebaby}
Let $(M, J, H)$ be a complex generalized complex manifold with nonexact twisting $H$. Then any generalized Hamiltonian torus action must have fixed point locus of dimension at least four.
\end{theorem}

\begin{proof}
Let $T$ be the torus and $\mu$ the moment map. By Proposition \ref{moment-map-Morse-Bott} and Corollary \ref{compmom} $(M,T, \mu)$ is an abstract moment map with compatible equivariant twisting $H + \alpha$ where $\alpha$ is the moment 1-form. Thus by Corollary \ref{qudimes}, $$\dim H(M ; H) = \dim H(M^T; H).$$ On the other hand, $0$ is always a compatible twisting so $$ \dim H(M) = \dim H(M^T).$$ Since $H$ is not exact, we know by (\ref{egadyikes}) that $\dim H(M; H) < \dim H(M)$. Consequently $H(M^T ; H) < H(M^T)$ and we conclude that the restriction of $H$ to $M^T$ is \emph{not} exact. Since $H$ is a 3-form, it must be that $M^T$ has a component of dimension 3 or more, and because $M^T$ is even dimensional (see Prop. \ref{fixed-points-submfold}) this completes the proof.
\end{proof}

Theorem \ref{applebaby} stands in stark contrast with the symplectic world, where Hamiltonian actions with isolated fixed points abound (e.g. toric manifolds). In the course of writing this paper, we discovered a proof of Theorem \ref{applebaby} that avoids twisted cohomology and in fact leads to even stronger constraints. We will present these arguments in a future paper, along with new examples of compact GC Hamiltonian actions.

\appendix

\section{Maximum principle for pseudo-holomorphic functions on almost complex manifolds}
\label{maximum-principle}

In this section, we give a self-contained proof of the maximum principle for pseudo-holomorphic functions on almost complex manifolds.
We believe that the maximum principle in this setting should have been known to experts working in the related areas and we are not claiming any originality. We are presenting a proof here just because the central results of our paper are built upon it for which we can not find a good reference.

Let $(M, J)$ be an almost complex manifold. Then the almost complex structure induces
a splitting of the complexified cotangent bundle $T^*_{\C}(M)=T^*(M)^{1,0}\oplus T^*(M)^{0,1}$.
In this context, a complex valued function $f+\frak{i}g \in C^{\infty}(M)$
is defined to be a pseudo-holomorphic function on $M$ if $(df)_x+\frak{i}(dg)_x \in T_x^*(M)^{1,0} $
for any $x\in M$. In this appendix, we prove the following result.

\begin{theorem} \label{maixmum-princ.-psedo-holomorphic-func.} Suppose $(M, J)$ is an almost complex manifold.
Suppose $f$ is the real part of a pseudo-holomorphic function $f+\frak{i}g$ on $M$.
Then for any $x\in M$ there exists an open neighborhood $B \ni x$ such that

\[ \sup_B f=\sup_{\partial B}f, \,\,\,\inf_B f=\inf_{\partial B}f.\]

\end{theorem}

Before beginning the proof, we first recall the maximum principle for the elliptic partial differential equations of second order
as treated in \cite{GT1977}.
Let $L$ be a second order linear differential operator on a domain $\Omega$ of $R^m$ given by
\begin{equation} \label{def-of-elliptic} Lu=a^{ij}D_{ij}u+b^iD_iu+c_iu,
\end{equation}
where $a^{ij}=a^{ji}$, $D_{i}u=\frac{\partial u}{\partial x_i}, D_{ij}u=\frac{\partial^2 u}{\partial x_i \partial x_j}$.
$L$ is said to be elliptic at $x \in \Omega$ if the coefficient matrix $[a^{ij}(x)]$ is positive definite.
$L$ is said to be elliptic in $\Omega$ if it is elliptic at each point of $\Omega$.
The following maximum principle is a fundamental result in the theory of elliptic operators.

\begin{theorem}\label{maximum-princ.-elliptic-eq}(\cite[Thm. 3.1]{GT1977}) Let $L$ be an elliptic operator in the bounded domain $\Omega$. Suppose that
\begin{equation} \label{sub-harmonic-condition}
Lu\geq 0(\leq 0) \text{ in } \Omega, \,\,\, c=0\,\text{ in }\Omega, \end{equation}
with $u \in C^2(\Omega) \cap C^0(\overline{\Omega})$. Then
\[\sup_{\Omega} u=\sup_{\partial \Omega}u, \,\,\,\inf_{\Omega} u=\inf_{\partial \Omega}u.\]

\end{theorem}

We are ready to present a proof of Theorem \ref{maixmum-princ.-psedo-holomorphic-func.}.

\begin{proof} Given an arbitrary point $p\in M$, we can choose a coordinate neighborhood $(U, x^1, x^2, \cdots, x^{2n})$ around $p$ such that under this coordinate system the almost complex $J(x)=[J^i_j(x)]$ coincides with the standard complex structure on $R^{2n}$ at the point $p$, i.e.,
\[J(p)=\left[ \begin{matrix} 0_n & I_n\\-I_n & 0_n \end{matrix} \right] ,\]
where $0_n$ denotes the $n\times n$ zero matrix and $I_n$ the $n\times n$ identity matrix.
Note that $\frac{\partial}{\partial x_k} +\frak{i}J\frac{\partial}{\partial x_k} \in
C^{\infty}(T(M)^{0,1})$ for any $ 1 \leq k \leq 2n$. We have
\[ < \frac{\partial}{\partial x_k} +\frak{i}J\frac{\partial}{\partial x_k}, df+\frak{i}dg>=0\]

since $df+\frak{i}dg \in C^{\infty}(T^*(M)^{1,0})$.

Observe $J\frac{\partial}{\partial x_k}=J_k^p\frac{\partial}{\partial x_p}$.
We get the following generalized Riemann-Cauchy equations. \begin{equation}\label{generalized-Riemann-Cauchy}
\frac{\partial f}{\partial x_k}=J_k^p\frac{\partial g}{\partial x_p},\,\,\,\,
\frac{\partial g}{\partial x_k}=-J_k^p\frac{\partial f}{\partial x_p}
\end{equation}

Therefore \[\begin{split} \frac{\partial}{\partial x_k}\left( \frac{\partial f}{\partial x_k}\right)&=
\frac{\partial}{\partial x_k}\left(J_k^p\frac{\partial g}{\partial x_p}\right)\\&=
\frac{\partial J_k^p}{\partial x_k}\frac{\partial g}{\partial x_p}+J_k^p\frac{\partial}{\partial x_p}\left(\frac{\partial g}{\partial x_k} \right)\\
&=\frac{\partial J_k^p}{\partial x_k}J_p^q\frac{\partial f}{\partial x_q}+J_k^p\frac{\partial}{\partial x_p}\left(-J_k^q \frac{\partial f}{\partial x_q}\right)\\&= \frac{\partial J_k^p}{\partial x_k}J_p^q\frac{\partial f}{\partial x_q}-J_k^p\frac{\partial J_k^q}{\partial x_p}\frac{\partial f}{\partial x_q}-J_k^pJ_k^q\frac{\partial^2 f }{\partial x_p \partial x_q}.
 \end{split}\]

 It follows \[  \frac{\partial^2 f}{\partial x^2_k}+J_k^pJ_k^q\frac{\partial^2 f }{\partial x_p \partial x_q}+\left(J_k^p\frac{\partial J_k^q}{\partial x_p}-\frac{\partial J_k^p}{\partial x_k}J_p^q \right)\frac{\partial f}{\partial x_q}=0.\]

 Summing over the index $k$ we get a second order linear equation  \begin{equation} \label{2nd-order-elliptic-equation} \sum_k\left(\frac{\partial^2 f}{\partial x^2_k}+J_k^pJ_k^q\frac{\partial^2 f }{\partial x_p \partial x_q}\right)+\sum_k\left(J_k^p\frac{\partial J_k^q}{\partial x_p}-\frac{\partial J_k^p}{\partial x_k}J_p^q \right)\frac{\partial f}{\partial x_q}=0.\end{equation}

 Set \[ a^{pq}= \delta_p^q+\sum_kJ_k^pJ_k^q,\] where $\delta_p^q$ is the Kronecker symbol.
 Since $J(p)$ coincides with the standard complex structure on $\R^{2n}$, a simple calculation shows that the matrix
 $[a^{pq}(p)]=2I_{2n}$, where $I_{2n}$ denotes the $2n \times 2n$ identity matrix. This shows clearly that
 the matrix $[a^{pq}]$ is a positive definite symmetric matrix at $p$ and so must be a positive definite symmetric matrix
 on an open ball  $p \in B \subset U$. Thus Equation \ref{2nd-order-elliptic-equation} is a second order elliptic
 equation on a bounded ball $B$. Now Theorem \ref{maixmum-princ.-psedo-holomorphic-func.} is a simple consequence of Theorem \ref{maximum-princ.-elliptic-eq}.  \end{proof}

\section{nondegenerate abstract moment maps}\label{nondegabs}

In this section we will prove Lemma \ref{immunology}. This result was used without much explanation in appendix G of \cite{GGK02} but was later recognized to be subtler than it appears. Here we provide a detailed proof using an extra hypothesis, the existence of an invariant almost (more generally stable) complex structure.

Throughout, let $T$ be a compact torus with lie algebra $\lie{t}$, $M$ be a smooth $T$-manifold, and $\phi: M \rightarrow \lie{t}^*$ a nondegerate, abstract moment map ( Def. \ref{absmom}).

As explained in \cite{GGK02}, condition 1 of Definition \ref{absmom} is equivalent to the condition that for all $ p \in M$,

\begin{equation}\label{Tuesone}
d \phi_p(T_pM) = \lie{t_p}^{\perp}.
\end{equation}
Here given a subspace $\lie{h}\subset \lie{t} $, $\lie{h}^{\perp} \subset \lie{t}^*$ denotes the annihilator of $\lie{h} \subset \lie{t}$.
That said, the next Lemma is not so surprising:

\begin{lemma}\label{Tuestwo}
Let $p,q \in M$ satisfy $\lie{t}_p = \lie{t}_q = \lie{g}$ where $\lie{g}$ is the Lie algebra of a subtorus $G \subset T$. If $p$ and $q$ lie in the same connected component of $ M^{G}$ then $$ \phi(p) + \lie{t}_p^{\perp} = \phi(q) + \lie{t}_q^{\perp}$$
\end{lemma}

\begin{proof}
The map $ \phi_{\lie{g}}: M \rightarrow \lie{g}^*$ defined by composing $\phi$ with the projection $\pi_{\lie{g}^*}: \lie{t}^* \rightarrow \lie{g}^*$ is a moment map for the restricted $G$ action. By definition, $\phi_{\lie{g}}$ restricts to a locally constant function on $M^G$ so $\phi_{\lie{g}} ( p) = \phi_{\lie{g}} (q)$ and so $\phi(p) - \phi(q) \in \ker( \pi_{\lie{g}^*}) = \lie{g}^{\perp}$.
\end{proof}

For compact $T$ manifold $M$, there can only be a finite number of distinct isotopy groups $T_p \subset T$, for each of which $M^{T_p}$ has a finite number of components. We deduce:

\begin{corollary}\label{vasy}
If $M$ is compact, then the set of vector spaces $$\{ span(\phi(p)) + \lie{t}_p^{\perp} | p \in M\}$$ is finite.
\end{corollary}

\begin{lemma}\label{pennyi}
Let $M$ be a compact $T$-manifold equipped with a nondegerate moment map $\phi$ and suppose that $0 \in \lie{t}^*$ is a regular value for $\phi$. There exists a codimension 1 subtorus $H \subset T$ with Lie algebra $\lie{h}$ for which 0 is a regular value for $\phi_{\lie{h}} = \pi_{\lie{h}^*} \circ \phi$, where $\pi_{\lie{h}^*} : \lie{t}^* \rightarrow \lie{h}^*$ is projection.
\end{lemma}

\begin{proof}
The hyperplane Grassmanian $Gr_1(\lie{t})$ parametrizes the set of codimension one subspaces $\lie{h} \subset \lie{t}$. Those subspaces integrating to codimension one subtori form a dense subset of $Gr_1(\lie{t})$, so to prove Lemma \ref{pennyi} it will suffice to show that the set $$U := \{ \lie{h} \in Gr_1(\lie{t}) | \text{ 0 is a regular value for } \phi_{\lie{h}} = proj_{\lie{h^*}} \circ \phi\}$$ contains a nonempty open set.

It is somewhat clearer to work with the projective space $P(\lie{t}^*)$, which is canonically isomorphic to $Gr_1(\lie{t})$ via the correspondence $\lie{h} \leftrightarrow \lie{h}^{\perp}$. Since $\lie{h}^{\perp} = \ker( \pi_{\lie{h}^*})$ it follows easily that $ 0$ is a regular value for $\phi_{\lie{h}}$ if and only if $$\lie{t}^* = im(d\phi_p) + \lie{h}^{\perp}  = \lie{t}_p^{\perp} + \lie{h}^{\perp}$$ for all $p \in M$ satisfying $ \phi( p) \in  \lie{h}^{\perp}$.

If $\phi(p) = 0$ then   $im(d\phi_p) = \lie{t}^*$ by hypothesis. If $\phi(p) \in \lie{h}^{\perp} - 0$, then $span( \phi(p)) = \lie{h}^{\perp}$. Thus if $\lie{h}^{\perp}$ lies outside of the finite set of \emph{proper} vector subspaces described in Corollary \ref{vasy}, the moment map $\phi_{\lie{h}}$ is guaranteed to be regular at zero. This is an open and nonempty condition, completing the proof.
\end{proof}

Iterating Lemma \ref{pennyi} enables us to prove (i) and (ii) of Proposition \ref{immunology}. In particular, we construct a sequence of subtori $T =T_n \supset T_{n-1} \supset T_{n-1} ....\supset T_1$ such that the moment map $\phi_k$ for each $T_k$ is regular at $0$, and then choose arbitrarily $ \xi_k \in \lie{t}_k - \lie{t}_{k-1}$.

We prove (iii) in two steps. We denote the restricted function $f_k:= \phi^{\xi_{k+1}}|_{\phi_k^{-1}(0)}$.

\begin{lemma}
The critical set $Crit(f_k) = \{ p \in \phi^{-1}_k(0) | \text{ $(\lie{t}_{k+1})_p \neq 0 $ } \}  $.
\end{lemma}

\begin{proof}
For $p \in M$ we have $$ \dim ( d_p f_k ( T_p \phi_k^{-1}(0))) = \dim ( d_p \phi_{k+1} ( T_p \phi_k^{-1}(0))) = \dim ( d_p \phi_{k+1} ( T_p M)) - k = 1 - \dim ( (\lie{t}_k)_p).$$ ere we have used Equality \ref{Tuesone}.
\end{proof}

It remains to prove that the critical points of $f_k:= \phi^{\xi_{k+1}}|_{\phi_k^{-1}(0)}$ are nondegenerate. According to \cite{GGK02}, if $M$ admits an invariant almost complex structure, then in the neighborhood of every orbit $M$ admits a symplectic structure for which the moment map is Hamiltonian. We can then use the following local canonical form \cite{GS82} for symplectic Hamiltonian actions.

\begin{lemma} (\cite{GS82})
Suppose that $M$ admits a $T$-invariant almost complex structure. For $p \in M$ choose a complimentary Lie subalgebra $\lie{h} \subset \lie{t}$ to $\lie{t}_p$ so that $ \lie{t}^* = \lie{h}^* \oplus \lie{t}_p^*$. Then for some $T_p$ representation $V$, there is a $T$-equivariant diffeomorphism from an invariant neighborhood of $p$ to an invariant neighborhood of the zero section of the associated bundle $ T \times_{T_p} ( \lie{h}^* \oplus V)$ sending the moment map $\phi$ to the map $$ \phi' : T \times_{T_p} ( \lie{h}^* \oplus V) \rightarrow  \lie{h}^* \oplus \lie{t}_p^* = \lie{t}^*$$ defined by $ \phi'(t, \eta, v) = (\eta, q(v))$, where $q: V \rightarrow \lie{t}_p^*$ is a quadratic form.
\end{lemma}

Applying this to the case $T = T_{k+1}$, $\phi = \phi_{k+1}$, $\lie{h} = \lie{t}_k$ at a critical point $p$ of $f_k$, we obtain a local model for $\phi_{k+1}$ near $p$ in $M$ $$ \phi': T_{k+1} \times_{(T_{k+1})_p} ( \lie{t}_k^* \oplus V) \rightarrow \lie{t}_k^* \oplus \R = \lie{t}_{k+1}^*$$ where $\phi'( t, \eta, v) = \eta + q(v)$. Here a neighborhood of $p$ in $\mu_k^{-1}(0)$ maps to $ T_{k+1} \times_{(T_{k+1})_p} ( \{0\}\oplus V)$ and $f_k$ corresponds (up to a nonzero scalar multiple) to the quadratic form $q$. Thus in some local coordinates, $f_k$ looks like a quadratic form near $p$ in $\phi^{-1}_k(0)$ and hence is nondegenerate, completing the proof of Lemma \ref{immunology}.

\section{ Further aspects of twisted equivariant cohomology}\label{apptwis}

We call a $G$-invariant open cover $\{ U_{\alpha}| \alpha \in I \}$ of $M$ equivariantly good if all nonempty intersections of the $U_{\alpha}$ are tubular neighborhoods of some $G$-orbit in $M$. For example, every $G$-equivariant vector bundle over a compact manifold $M$ admits a \emph{finite} equivariantly good cover.

\begin{lemma}\label{fouture}
Let $M$ be a $G$-manifold admitting a finite equivariantly good open cover $\{ U_{\alpha}| \alpha \in I \}$. Then for all $J \subset I$, and for any twisting $d_G$-closed 3-form $\eta \in \Omega_G^3(M)$ we have an isomorphism of $H_G(M)$-modules
$$ H_G( \cap_{\alpha \in J} U_{\alpha} ; \eta) \cong H(BH)$$
for some closed subgroup $H \subset G$.
\end{lemma}

\begin{proof}
By definition we know that $ \cap_{\alpha \in J} U_{\alpha} $ is equivariantly homotopy equivalent to a homogeneous space $G/H$ for some subgroup $H \subset G$. Thus $\eta$, $H_G( \cap_{\alpha \in J} U_{\alpha} ) \cong H(G/H)$. But it is a standard result that $H_G(G/H) \cong H(BH)$ and $H^3(BH) = 0$ (c.f. \cite{AB84}).  Thus $\eta$ is cohomologous to zero. It follows that

$$ H_G( \cap_{\alpha \in J} U_{\alpha}; \eta )  \cong H_G( \cap_{\alpha \in J} U_{\alpha}) \cong H(BH) $$
\end{proof}

Lemma \ref{fouture} can be used to prove twisted equivariant cohomology results using Mayer-Vietoris.

\begin{proposition}
Let $M$ be a smooth $G$-manifold admitting a finite, equivariantly good open cover. Then for any twisting $\eta$, $H_G(M ;\eta)$ is a finitely generated $(\hat{S}\lie{g}^*)^G$-module.
\end{proposition}

\begin{proof}
For any closed subgroup $H \subset G$, $H(BH)$ is finitely generated over $H(BG) \cong (\hat{S}\lie{g}^*)^G$ (the number of generators is bounded by the Weyl group). Then $H_G(M ;\eta)$ is shown to be finitely generated by repeated application of Mayer-Vietoris and Lemma \ref{shex}.
\end{proof}

There a couple of different versions of twisted equivariant cohomology described in the literature and we take a moment to reconcile them. In \cite{FHT02}, it is defined as the cohomology of the complex of formal Laurent series $ \Omega_G^*(M)((\beta))$  where $\beta$ has degree $-2$ with differential $ d + \eta \beta$. This makes  $ \Omega_G^*(M)((\beta))$ into a graded complex, producing $\Z$-graded cohomology $\tilde{H}_G^*(M;\eta)$. The reader can readily verify that $ \sum_{k} \alpha^{n+2k} \otimes \beta^k \in \Omega^n_G(M)((\beta))$ is closed (exact) if and only if $ \sum_k \alpha^{n+k} \in \hat{\Omega}_G(M)$ is closed (exact) for $d_G + \eta$. It follows that $$\tilde{H}_G^n(M) \cong H_G^{[n]}(M)$$
where on the $n$ is an integer and $[n]$ is its reduction mod $2$.

Another version comes from \cite{HuUribe06}. They define $\bar{H}_G(M;\eta)$ to be the cohomology of the complex $ (\Omega(M) \otimes \hat{S} \lie{g}^*)$ with differential $ d_G + \eta \wedge$. There is a natural injective chain map

$$ (\Omega(M) \otimes \hat{S} \lie{g}^*) \hookrightarrow \hat{\Omega}_G(M)$$
which induces an isomorphism in cohomology in most interesting cases:

\begin{proposition}\label{tousecof}
Let $M$ be a smooth $G$-manifold admitting a finite, equivariantly good open cover. Then $\bar{H}_G(M) \cong H_G(M)$.
\end{proposition}

 \begin{proof}
Using Equation \ref{homprod}, it is easy to prove the isomorphism holds on tubular neighborhoods of $G$-orbits. This extends to $M$ by repeated application of Mayer-Vietoris.
 \end{proof}

This result allows us to compare $H_G(M;\eta)$ with the cohomology of the direct sum complex $( \Omega_G(M) = \bigoplus_i \Omega_G^i(M), d_G + \eta \wedge )$.

\begin{proposition}
For $M$ admitting a finite equivariantly good cover, $H_G(M; \eta)$ is canonically isomorphic to $ H( \Omega_G(M), d_G + \eta) \otimes_{(S \lie{g}^*)^G} (\hat{S}\lie{g}^*)^G$ (i.e. it is obtained by extension of scalars).
\end{proposition}
\begin{proof}
By Proposition \ref{tousecof} $H_G(M;\eta) \cong \bar{H}_G(M ;\eta)$. The formal power series ring $ \hat{S} \lie{g}^*$ is flat over $S \lie{g}$ so it factors through taking cohomology (see Lemma \ref{b6}). Thus $$ \bar{H}_G(M;\eta) = H( \Omega_G(M) \otimes_{(S \lie{g}^*)^G} (\hat{S} \lie{g}^*)^G; d + \eta ) = H( \Omega_G(M) ; d + \eta )\otimes_{(S \lie{g}^*)^G} (\hat{S} \lie{g}^*)^G.$$
 \end{proof}

Recall (\cite{Mc01} Def. 3.8) that given a filtration $ \{ F^p A\}$ of a differential complex $(A,d)$, the associated spectral sequence $(E_r^p, d_r)$ is said to \textbf{converge strongly} to $H(A,d)$ if the induced maps $F^pH(A,d)/ F^{p+1}H(A,d) \rightarrow E_{\infty}^p$ and $ H(A,d) \rightarrow (\stackrel{lim}{\leftarrow} H(A,d)/F^pH(A,d)$ are isomorphisms. The first of these

\begin{proposition}\label{jiggsdiner}
Let $M$ admit a finite $G$-equivariantly good open cover. Then the spectral sequences associated to the filtrations $\{ F^p\}$ and $\{L^p\}$ of $(\hat{\Omega}_G(M), d_{G,\eta})$ described in \S \ref{Spectral Sequences} both converge strongly to $ H_G(M ;\eta)$.
\end{proposition}

\begin{proof}
First note that the filtrations are cofinal, satisfying

$$ F^{2p-n} \subset L^p \subset F^{2p+n}, $$
so convergence of the spectral sequence determined by $\{F^p\}$ implies convergence of that for $\{ L^p\}$.

Clearly $\cup_p F^p = \hat{\Omega}_G(M) and \cap_p F^p = 0$ so the filtration is exhaustive and weakly convergent. By ( \cite{Mc01} Thm 3.2) it only remains to show that we have an isomorphism
\begin{equation}
\label{completeness} H_G(M;\eta) \cong \stackrel{lim}{\leftarrow} H_G(M;\eta)/F^pH_G(M;\eta).
\end{equation}

First note that if $\eta =0$, then the filtration $F^p H_G(M) = \prod_{k \geq p} H_G^k(M)$ so (\ref{completeness}) certainly holds. Next, note that wedging by $exp(b)$ for any $ b \in \Omega_G^2(M)$ preserves the filtration $F^p$, so by Lemma \ref{carumba}, (\ref{completeness}) must hold for exact $\eta$. Finally, because $\eta$ must become exact when restricted to tubular neighborhoods of orbits, we may use Mayer-Vietoris and the five lemma to prove the general case.
\end{proof}

Proposition \ref{jiggsdiner} was proven with more general hypotheses in \cite{FHT02} using a Mittag-Lefler condition. We include our own proof because it is more elementary.

\section{Commutative algebra}\label{commalg}

The results in this section are standard. Our principal reference is Matsumura \cite{Mat89} Chapter 8.

Recall that a module is Noetherian is it satisfies the ascending chain condition. A ring is Noetherian if it is Noetherian over itself. Throughout, $R$ will denote a Noetherian, commutative integral domain.

\begin{lemma}
 Let $M$ be an $R$-module. The following are equivalent:\\
(1) $M$ is Noetherian\\
(2) $M$ is finitely generated\\
(3) $M$ has a finite presentation
\end{lemma}

\begin{lemma}\label{shex}
For a short exact sequence of $R$-modules $$0 \rightarrow K \rightarrow M \rightarrow N \rightarrow 0$$ $M$ is Noetherian if and only if both $K$ and $N$ are Noetherian.
\end{lemma}

An $R$-module $M$ is called flat if the functor $ M \otimes_R : Mod_R \mapsto Ab$ from $R$-modules to abelian groups is exact.

\begin{lemma}
For $R$ an Noetherian, integral domain, the following are flat over $R$:\\
(1) the quotient field of $R$.\\
(2) any $I$-adic completion of $R$, for $I \subset R$ an ideal.\\
(3) any projective module over $R$.
\end{lemma}

Recall that for $I \subset R$ an ideal and $M$ an $R$-module, the $I$-adic completion of $M$ is defined by the inverse limit: $$ \hat{M} = \stackrel{lim}{\leftarrow} M/ I^k M$$
We have $\hat{R}$ is a ring and $\hat{M}$ is a $\hat{R}$-module.

\begin{lemma}\label{b4}
If $M$ is finitely generated $R$-module, then there is a natural isomorphism of $\hat{R}$-modules: $$ \hat{M} \cong M \otimes_R \hat{R}$$
\end{lemma}

Recall that a $R$-chain complex $(C,d)$ is an $R$-module $C$, equipped with a morphism $d: C \rightarrow C$ satisfying $d^2=0$. The homology is defined $H(C,d) = \ker d / \im d$ and is an $R$-module. Given a module $M$ over $R$, the tensor product $( C \otimes M, d \otimes 1)$ is a $\Z$-chain complex.

\begin{lemma}\label{b6}(ex. 7.6 in \cite{Mat89})
Let $(C,d)$ be an $R$-chain complex and let $M$ be flat over $R$. Then we have a natural isomorphism $H(C,d)\otimes_R M \cong H(C \otimes_R M, d \otimes_R 1)$.
\end{lemma}

The Jacobson radical is the ideal $$ J := \{ r \in R | 1 - rs\text{ is a unit in R for all }s \in S\}.$$

\begin{lemma} (Nakayama's Lemma)
Let $M$ be a nonzero finitely generated module over $R$ and $J$ the Jacobson radical of $R$. Then $JM \neq M$.
\end{lemma}
We remark that Nakayama's Lemma works for not necessarily commutative rings (see 4.3.10  \cite{We94}) so we can apply it to super commutative rings like $H_T(M)$.

We have some particular examples in mind. For instance the polynomial ring $A = \R[x_1,...,x_n]$. If $I = (x_1,...,x_n) \subset A$ is the augmentation ideal, then the $I$-adic completion is $\hat{A} = \R[[x_1,...,x_n]]$, the ring of formal power series, so $\hat{A}$ is flat over $A$. Both $A$ and $\hat{A}$ are Noetherian commutative integral domains so their quotient fields are flat over each of them.

\end{document}